\documentclass[11pt]{article}%{ctexart}
\usepackage{graphicx}
\usepackage[margin=0.5in]{geometry}
\usepackage{bm}
\usepackage{amsthm}
\usepackage{amscd}
\usepackage{amsbsy}
\usepackage{amsmath}
\usepackage{amsfonts}
\usepackage{amssymb}
\usepackage{calligra}
\usepackage{xcolor}
\usepackage{float}
\usepackage[T1]{fontenc}
\usepackage{multirow}
\usepackage{mathrsfs}
\usepackage{mathtools}
\usepackage{epstopdf}
\usepackage{psfrag}
\usepackage{subfigure}
\usepackage{booktabs}
\usepackage{listings}
\usepackage{longtable}
\usepackage{latexsym}
\usepackage{arydshln}
\usepackage{enumerate}
\usepackage{epsfig}
\usepackage{cite}
\usepackage{verbatim}
\usepackage{overpic}
\usepackage{abstract}
\allowdisplaybreaks[4]

\date{}
\title{New sufficient condition for the two-dimensional real Jacobian conjecture through the Newton diagram}
\author{ Yuzhou Tian$^{\text{a}}$ and Xiuli Cen$^{\text{b},}$\footnote{Corresponding author. }\\
	\it\footnotesize Department of Mathematical Sciences, Tsinghua University, Beijing 100084, China$^{\text{a}}$, \\
	\it\footnotesize E-mail: tianyuzhou2016@163.com\\
	\it\footnotesize School of Mathematics and Statistics, HNP-LAMA, Central South University, Changsha,
Hunan 410083, China$^{\text{b}}$\\
	\it\footnotesize E-mail: cenxiuli2010@163.com}

\newtheorem {theorem*}{Theorem}
\newtheorem {theorem} {Theorem}
\newtheorem{example}{Example}

\newtheorem{proposition}{Proposition}
\newtheorem{lemma}{Lemma}
\newtheorem{corollary}{Corollary}
\newtheorem{remark}{Remark}

\numberwithin{equation}{section}

\usepackage[colorlinks,
            CJKbookmarks=true,
            linkcolor=blue,
            anchorcolor=red,
             urlcolor=blue,
            citecolor=blue
            ]{hyperref}

\begin{document}
\maketitle
\noindent {\bf Abstract} The present paper is devoted to investigating the two-dimensional real Jacobian conjecture. This conjecture claims that if $F=\left(f,g\right):\mathbb{R}^2\rightarrow \mathbb{R}^2$ is a polynomial map with $\det DF\left(x,y\right)\ne0$ for all  $\left(x,y\right)\in\mathbb{R}^2$, then $F$ is globally injective.  With the help of the Newton diagram, we provide a new sufficient condition such that the two-dimensional real Jacobian conjecture holds.  Moreover, this sufficient condition generalizes the main result of [J. Differential Equations {\bf 260} (2016), 5250-5258]. Furthermore, two new classes of polynomial maps satisfying the two-dimensional real Jacobian conjecture are given.

\smallskip

\noindent {\bf 2020 Math Subject Classification: } Primary  34C05.  Secondary 34C08.  Tertiary 14R15

\smallskip


\section{Introduction and statement of the main result}\label{se-1}
Consider a smooth map
\begin{align*}
&F\left(x,y\right)=\left(f\left(x,y\right),g\left(x,y\right)\right):\mathbb{R}^2\rightarrow\mathbb{R}^2
\end{align*}
with $\det DF\left(x,y\right)\neq0$  for all $(x,y)\in\mathbb{R}^2$, where $f\left(x,y\right)$ and $g\left(x,y\right)$ are smooth functions.  As we know, the smooth map $F$ is a local diffeomorphism, but
it is not always globally injective in $\mathbb{R}^2$.  Generally, the smooth map $F$ is a global diffeomorphism under some restricted conditions,  %{\textcolor{red}{In fact, the smooth map $F$ is a global diffeomorphism only when some appropriate conditions are satisfied}},
see \cite{cobo2002injectivity,fernandes2004global,Braun2017On}.

 In 1939, Keller restricted his attention to the polynomial map $F$ in $\mathbb{C}^2$ and proposed a conjecture:  if $F:\mathbb{C}^2\rightarrow\mathbb{C}^2$ is a polynomial map such that  $\det DF\left(x,y\right)\equiv\text{constant}\neq0$ for all  $\left(x,y\right)\in\mathbb{C}^2$, then $F$ is globally injective.  This conjecture  is well-known \emph{Jacobian conjecture}.  In 1998, Smale \cite{MR1631413} selected $18$ great mathematical problems for the 21th century, and listed Jacobian conjecture as the 16th  problem.  Up to now, it is still unsettled  and is notoriously difficult.  Fortunately, there has been quite a bit of progress, see for example  \cite{bass1982jacobian, MR714106, MR714105,MR1487631,MR3866897, MR4242820}, etc.

Randall \cite{MR713265} considered the polynomial map $F$ in $\mathbb{R}^2$ and he gave another  conjecture:
if $F:\mathbb{R}^2\rightarrow\mathbb{R}^2$ is a polynomial map such that $\det DF\left(x,y\right)\ne0$ for all  $\left(x,y\right)\in\mathbb{R}^2$, then $F$ is globally injective.  This conjecture is known as \emph{real Jacobian conjecture}. However,  Pinchuk \cite{MR1292168} in 1994 constructed a counterexample for the real Jacobian conjecture.  Pinchuck's  counterexample leads to an interesting problem, i.e., to  find  some suitable  conditions  such that the real Jacobian conjecture holds.  This  problem has gained wide attention  from mathematicians. A stream of important results have been obtained in recent years. With the help of the structure of polynomial maps, the authors \cite{cima1995global,MR1362759} gave some sufficient conditions.  Gwo\'{z}dziewicz in \cite{MR1839866} showed that the real Jacobian conjecture holds if $\text{deg}\;f\leq3$ and $\text{deg}\;g\leq3$.  Afterwards Braun et al. \cite{MR2552779,MR3514314} improved this result and proved that the conjecture is correct if $\text{deg}\;f\leq4$, independent of $\text{deg}\;g$.
In the above works, the main tools involve algebra, analysis and geometry. The following global dynamical result established by Sabatini in \cite{MR1636592} allows us to explore the real Jacobian conjecture by some dynamical systems  techniques.
\begin{theorem}\label{th-1}{\rm (see \cite{MR1636592})}
	Let $F=\left(f,g\right):\mathbb{R}^2\rightarrow \mathbb{R}^2$ be a polynomial map such that $F\left(0,0\right)=\left(0,0\right)$ and $\det DF\left(x,y\right)\ne0$ for all  $\left(x,y\right)\in\mathbb{R}^2$. Then the following statements are equivalent.
	\begin{itemize}
		\item [\emph{(a)}] The origin is a global center for the Hamiltonian polynomial vector field
		\begin{align}\label{eq-1}
			&\mathcal{X}=\left(-ff_y-gg_y,ff_x+gg_x\right)\triangleq\left(\Lambda\left(x,y\right),\Omega\left(x,y\right)\right).
		\end{align}
		\item [\emph{(b)}] $F$ is a global diffeomorphism of the plane onto itself.
	\end{itemize}
\end{theorem}

Following Theorem \ref{th-1},  a lot of novel sufficient conditions for the validity of the real Jacobian conjecture are obtained by using the qualitative theory of dynamical systems.
For instance, the authors  in \cite{MR3448779} gave an important result as follows.
\begin{theorem}[see \cite{MR3448779}]\label{th3}
Let $F=\left(f,g\right):\mathbb{R}^2\rightarrow \mathbb{R}^2$ be a polynomial map such that $F\left(0,0\right)=\left(0,0\right)$ and $\det DF\left(x,y\right)\ne0$ for all  $\left(x,y\right)\in\mathbb{R}^2$. If the higher homogeneous terms of the polynomials $ff_x+gg_x$ and $ff_y+gg_y$ do not have real linear factors in common, then $F$ is a global injective.
\end{theorem}
\noindent  Notice that the sufficient condition given by Theorem \ref{th3} is not necessary
\cite{MR3448779, MR4219338}, thus the paper \cite{MR4219338} provided new sufficient conditions for the validity of the real Jacobian conjecture in which the higher homogeneous terms of the polynomials
$ff_x + gg_x$ and $ff_y + gg_y$ have real linear factors in common of multiplicity one.
Other references on this topic include the papers \cite{MR3443401,Jackson,tian2021necessary,MR4213674,vall2021}.

In algebraic geometry,  the \emph{Newton diagram}  is a powerful tool for investigating the behaviour of polynomials over local fields, see Section \ref{se-2} for the definition of the Newton diagram.  By algebraic methods, the scholars established a connection between the Jacobian conjecture and the Newton diagram,  and used this  connection to obtain several  valuable results, as shown in Chapter $10$ of  \cite{van2012polynomial}. So, it is natural to ask what happens when the real Jacobian conjecture encounters the Newton diagram.  %So, it is natural to ask how to apply the Newton diagram to study the real Jacobian conjecture.

In this work, we introduce the Newton diagram to present  a new sufficient condition such that the real Jacobian conjecture holds.  The qualitative theory of dynamical systems plays a crucial role in our proofs.

Let $d=\max\left\{\text{deg}\;\Lambda\left(x,y\right),\text{deg}\;\Omega\left(x,y\right)\right\}$. Denote by $b\left(\mathcal{X}\right)$ the \emph{Bendixson compactification}  of the Hamiltonian vector field $\mathcal{X}$ defined in \eqref{eq-1}.  For more details about Bendixson compactification, we refer to
Chapter $13$ of \cite{MR0350126} and Chapter $5$ of \cite{Dumortier2006qualitative}. The explicit expression for $b\left(\mathcal{X}\right)$ is
\begin{align}\label{eq-5}
	\begin{cases}
		\dot{u}=\left(u^2+v^2\right)^d\left[\left(v^2-u^2\right)\Lambda\left(\dfrac{u}{u^2+v^2},\dfrac{v}{u^2+v^2}\right)-2uv\Omega\left(\dfrac{u}{u^2+v^2},\dfrac{v}{u^2+v^2}\right)\right],\\
		\specialrule{0em}{3pt}{3pt}
		\dot{v}=\left(u^2+v^2\right)^d\left[\left(u^2-v^2\right)\Omega\left(\dfrac{u}{u^2+v^2},\dfrac{v}{u^2+v^2}\right)-2uv\Lambda\left(\dfrac{u}{u^2+v^2},\dfrac{v}{u^2+v^2}\right)\right].
	\end{cases}
\end{align}

Our main result is described as follows.
\begin{theorem}\label{th-11}
 Consider a polynomial map  $F=\left(f\left(x,y\right),g\left(x,y\right)\right):\mathbb{R}^2\rightarrow \mathbb{R}^2$ such that $F\left(0,0\right)=\left(0,0\right)$ and  $\det DF\left(x,y\right)\neq0$ for all  $\left(x,y\right)\in\mathbb{R}^2$. Denote by $b\left(\mathcal{X}\right)$ the Bendixson compactification of the Hamiltonian vector field $\mathcal{X}$ given in \eqref{eq-1}.
	Let $\mathcal{N}\left(b\left(\mathcal{X}\right)\right)$ be the Newton diagram of $b\left(\mathcal{X}\right)$. For each bounded edge of type $\mathbf{t}$ in $\mathcal{N}\left(b\left(\mathcal{X}\right)\right)$, if its associated Hamiltonian does not have any factor of the form $v^{t_1}-\lambda u^{t_2}$ with $\mathbf{t}= \left(t_1,t_2\right)\in\mathbb{N}^2$ and $\lambda\in\mathbb{R}\setminus\{0\}$, then $F$ is injective.
\end{theorem}

%\begin{remark}
% \emph{In Section \ref{se-4}, we will give an example to show that Theorem \ref{th-11} improves the results in \cite{MR3448779} and \cite{MR4219338} }.
%\end{remark}

The following theorem reveals the relationship between Theorems \ref{th-11} and \ref{th3}, which shows that the sufficient condition in Theorem \ref{th-11} is weaker than the sufficient condition in Theorem \ref{th3}. Thus, our
result generalizes the main result of \cite{MR3448779}.
\begin{theorem}\label{th2}
 Consider a polynomial map  $F=\left(f\left(x,y\right),g\left(x,y\right)\right):\mathbb{R}^2\rightarrow \mathbb{R}^2$ such that $F\left(0,0\right)=\left(0,0\right)$ and  $\det DF\left(x,y\right)\neq0$ for all  $\left(x,y\right)\in\mathbb{R}^2$.
 Denote by $b\left(\mathcal{X}\right)$ the Bendixson compactification of the Hamiltonian vector field $\mathcal{X}$ given in \eqref{eq-1}, and by $\mathcal{N}\left(b\left(\mathcal{X}\right)\right)$ the Newton diagram of $b\left(\mathcal{X}\right)$. If the higher homogeneous terms of the polynomials $ff_x+gg_x$ and $ff_y+gg_y$ do not have real linear factors in common, then  for each bounded edge of type $\mathbf{t}$ in $\mathcal{N}\left(b\left(\mathcal{X}\right)\right)$, its associated Hamiltonian does not have any factor of the form $v^{t_1}-\lambda u^{t_2}$ with $\mathbf{t}= \left(t_1,t_2\right)\in\mathbb{N}^2$ and $\lambda\in\mathbb{R}\setminus\{0\}$.
\end{theorem}
As the applications of Theorem \ref{th-11}, two new classes of polynomial maps satisfying the real Jacobian conjecture in $\mathbb{R}^2$ are given, in which cases the higher homogeneous terms of the polynomials $ff_x+gg_x$ and $ff_y+gg_y$ have real linear factors in common of multiplicity more than one, thus the verification of these examples can not be tackled using the existing methods in \cite{MR3448779} and \cite{MR4219338}.
\begin{example}\label{ex1} {\rm Let}
$$F=\left(f,g\right)=\left(\sum_{i=0}^na_ix^{2i+1},y+\sum_{i=1}^mb_ix^{2i}\right),$$
{\rm where $a_i\geq0,\ i=0,1,\ldots,n$, $b_i\geq0,\ i=1,2,\ldots,m$, and $a_0>0, a_{n}>0, b_m>0, n\geq m\geq1$.
The map $F$ is injective.

Clearly, $\det DF\left(x,y\right)=\sum_{i=0}^n(2i+1)a_ix^{2i}>0$.  The Hamiltonian vector
 field \eqref{eq-1} associated to $F$  is
\begin{equation*}\begin{split}
\mathcal{X}&=\left(-ff_y-gg_y,ff_x+gg_x\right)\\
&=\left(-y-\sum_{i=1}^mb_ix^{2i},
y\sum_{i=1}^m2ib_ix^{2i-1}
+\sum_{i=1}^mb_ix^{2i}\sum_{i=1}^{m}2ib_ix^{2i-1}+\sum_{i=0}^na_ix^{2i+1}\sum_{i=0}^n(2i+1)a_ix^{2i}\right).
\end{split}\end{equation*}
The higher
homogeneous terms of $ff_x+gg_x$ and $ff_y+gg_y$  have the linear common factor $x$ with multiplicity $2m\geq2$. Thus, neither Theorem \ref{th3} nor Theorem 2 in \cite{MR4219338} does work for this example.}
\end{example}
\begin{example}\label{ex2}{\rm Let}
$$F=\left(f,g\right)=\left(\sum_{i=0}^{m_1}a_iy^{2i+1}+\sum_{i=0}^{m_2}b_ix^{2i+1},
\sum_{i=0}^{m_3}c_iy^{2i+1}-\sum_{i=0}^{m_2}d_ix^{2i+1}\right),$$
{\rm where $a_i\geq0,\ i=0,1,\ldots,m_1$, $b_i\geq0,\ i=0,1,\ldots,m_2$, $c_i\geq0,\ i=0,1,\ldots,m_3$,
$d_i\geq0,\ i=0,1,\ldots,m_2$, and $b_0c_0+a_0d_0>0, a_{m_1}>0, b_{m_2}>0, d_{m_2}>0, m_1>\max\left\{m_2,m_3\right\}\geq0$.
The map $F$ is injective.

Obviously, $$\det DF\left(x,y\right)=\sum_{i=0}^{m_2}(2i+1)b_ix^{2i}\sum_{i=0}^{m_3}(2i+1)c_iy^{2i}+
\sum_{i=0}^{m_1}(2i+1)a_iy^{2i}\sum_{i=0}^{m_2}(2i+1)d_ix^{2i}\geq b_0c_0+a_0d_0>0.$$  The Hamiltonian vector
 field \eqref{eq-1} associated to $F$  is
{\small{\begin{equation*}\begin{split}
\mathcal{X}=&\left(-ff_y-gg_y,ff_x+gg_x\right)\\
=&\left(-\left(\sum_{i=0}^{m_1}a_iy^{2i+1}+\sum_{i=0}^{m_2}b_ix^{2i+1}\right)\sum_{i=0}^{m_1}(2i+1)a_iy^{2i}
-\left(\sum_{i=0}^{m_3}c_iy^{2i+1}-\sum_{i=0}^{m_2}d_ix^{2i+1}\right)\sum_{i=0}^{m_3}(2i+1)c_iy^{2i},\right.\\
&\left.\left(\sum_{i=0}^{m_1}a_iy^{2i+1}+\sum_{i=0}^{m_2}b_ix^{2i+1}\right)\sum_{i=0}^{m_2}(2i+1)b_ix^{2i}
-\left(\sum_{i=0}^{m_3}c_iy^{2i+1}-\sum_{i=0}^{m_2}d_ix^{2i+1}\right)\sum_{i=0}^{m_2}(2i+1)d_ix^{2i}\right).
\end{split}\end{equation*}
}}It is easy to verify that the higher
homogeneous terms of $ff_x+gg_x$ and $ff_y+gg_y$  have the linear common factor $y$ with multiplicity $2m_1+1>2$. Thus, Theorem \ref{th3} and Theorem 2 in \cite{MR4219338} also can not be applied to this example.}
\end{example}
\begin{remark}
\emph{By Theorem \ref{th-1}, Examples \ref{ex1} and \ref{ex2} present two new classes of polynomial Hamiltonian vector fields with a global center. For the Hamiltonian vector fields, the characterization of the global center is a particularly challenging problem due to the fact that it relates to the  complex global dynamical analysis. Constructing such Hamiltonian vector fields has great significance to find limit cycles of planar polynomial vector fields by bifurcation theory.}
\end{remark}
The paper is organized as follows. Some preliminary definitions and results will be given in Section \ref{se-2}. In Section \ref{se-3}, we prove Theorems \ref{th-11} and \ref{th2}. At last, we solve these two examples to illustrate the practicability of our main result in Section \ref{se-4}.
%The last section will further discuss that map $F$ does not satisfy Theorem \ref{th-11}, that is, the case of existing factor of the form $v^{t_1}-\lambda u^{t_2}$ with $\lambda$ non-zero real.

\section{Preliminary definitions and results}\label{se-2}
In this section, we recall some basic  definitions and known results for the proof of Theorem \ref{th-11}.

\subsection{Newton diagram}\label{sub-1}

We consider the  following  analytic autonomous differential system
\begin{align}\label{eq-3}
	&\dot{x}=P\left(x,y\right),\quad\dot{y}=Q\left(x,y\right).
\end{align}
The vector field $\mathscr{X}$ associated to system \eqref{eq-3} is defined by $\mathscr{X}=\left(P,Q\right)$ or
\begin{align*}
	&\mathscr{X}=P\left(x,y\right)\frac{\partial}{\partial x}+Q\left(x,y\right)\frac{\partial}{\partial y}.
\end{align*}

Let $\text{\bf t}=\left(t_1,t_2\right)\neq\mathbf{0}$ with $t_1$ and $t_2$ coprime non-negative integers. A polynomial $R\left(x,y\right)$ is \emph{quasi-homogeneous} of type $\mathbf{t}$ and degree $k$ if $R\left(\lambda^{t_1}x,\lambda^{t_2}y\right)=\lambda^kR\left(x,y\right)$ for all $\lambda\in\mathbb{R}^+$. We denote by $\text{deg}_xR$ and $\text{deg}_yR$ the degree of the polynomial $R\left(x,y\right)$ with respect to $x$ and $y$, respectively.  The vector
space of quasi-homogeneous polynomials of type $\mathbf{t}$ and degree $k$ is denoted by $\mathcal{P}_{k}^{\mathbf{t}}$. The vector field $\mathscr{X}_k=\left(P_{k+t_1},Q_{k+t_2}\right)$ is \emph{quasi-homogeneous} of type $\mathbf{t}$ and degree $k$ if $P_{k+t_1}\in \mathcal{P}_{k+t_1}^{\mathbf{t}}$ and $Q_{k+t_2}\in \mathcal{P}_{k+t_2}^{\mathbf{t}}$. We denote the vector space of the quasi-homogeneous polynomial vector fields of type $\mathbf{t}$ and degree $k$ by $\mathcal{Q}_k^{\mathbf{t}}$.

Any analytic vector field $ \mathscr{X}$ can be written as  the sum of its  quasi-homogeneous components, that is,
\begin{align*}
&\mathscr{X}=\mathscr{X}_r+\mathscr{X}_{r+1}+\cdots=\sum_{k=r}^\infty\mathscr{X}_k,
\end{align*}
where $r\in \mathbb{N}$, $\mathscr{X}_r \not\equiv0$ and $\mathscr{X}_k=\left(P_{k+t_1},Q_{k+t_2}\right)\in \mathcal{Q}_k^{\mathbf{t}}$. For every quasi-homogeneous vector field $\mathscr{X}_k\in \mathcal{Q}_k^{\mathbf{t}}$, it can be expressed as
\begin{align}\label{eq-16}
&\mathscr{X}_k=\mathbf{X}_{h_{k+|\text{\bf t}|}}+\mu_k \mathbf{D}_0,
\end{align}
where
\begin{align*}
	&|\text{\bf t}|=t_1+t_2,\;\mathbf{D}_0:=\left(t_1x,t_2y\right),\;\mu_k:=\frac{1}{k+|\text{\bf t}|}\text{div}\left(\mathscr{X}_k\right)\in \mathcal{P}_{k}^{\mathbf{t}},\\
	&h_{k+|\text{\bf t}|}:=\frac{1}{k+|\text{\bf t}|}\left(\mathbf{D}_0\wedge\mathscr{X}_k\right)=\frac{1}{k+|\text{\bf t}|}\left(t_1xQ_{k+t_2}-t_2yP_{k+t_1}\right)\in \mathcal{P}_{k+|\text{\bf t}|}^{\mathbf{t}},\\
	&\mathbf{X}_{h_{k+|\text{\bf t}|}}:=\left(-\frac{\partial h_{k+|\text{\bf t}|}}{\partial y},\frac{\partial h_{k+|\text{\bf t}|}}{\partial x}\right),
\end{align*}
see for instance \cite{algaba2009integrability}.  Equation \eqref{eq-16} is  the classical \emph{conservative-dissipative splitting of a quasi-homogeneous vector field}. The quasi-homogeneous polynomial $h_{k+|\text{\bf t}|}$ is called the \emph{Hamiltonian  associated to $\mathscr{X}_k$}.

We briefly introduce  the Newton diagram as follows, see \cite{MR2819283,berezovskaya1993complicated,brunella1990topological,medvedeva2006analytic} for more details. Let $yP\left(x,y\right)=\sum a_{ij}x^iy^j$, $xQ\left(x,y\right)=\sum b_{ij}x^iy^j$ and $R\left(x,y\right)=\sum c_{ij}x^iy^j$. The \emph{support} of $\mathscr{X}$ is defined to be
$$\text{supp}\left(\mathscr{X}\right)=\left\{\left(i,j\right)\;|\;\left(a_{ij},b_{ij}\right)\neq\mathbf{0}\right\}\subset\mathbb{R}^2.$$
The vector $\left(a_{ij},b_{ij}\right)$  is called the \emph{vector coefficient} of $\left(i,j\right)$ in the support. The set
$$\text{supp}\left(R\right)=\left\{\left(i,j\right)\;|\;c_{ij}\neq0\right\}\subset\mathbb{R}^2$$
 is also said to be the \emph{support} of the polynomial $R\left(x,y\right)$.

Consider the convex hull $\Gamma$ of the set
\begin{align}\label{eq-14}
&\bigcup\limits_{\left(i,j\right)\in\text{supp}\left(\mathscr{X}\right)}\left(\left(i,j\right)+\mathbb{R}_+^2\right),
\end{align}
where $\mathbb{R}_+^2$ is the positive quadrant. The boundary of the convex hull $\Gamma$ consists of two open rays and a polygonal line that may be just a single point. This polygonal line is called the \emph{Newton diagram} of the vector field $\mathscr{X}$. Analogously to the above definitions, we can define the \emph{Newton diagram} of the  polynomial $R\left(x,y\right)$.

We say that the line segments of the polygonal line are \emph{edges} of the Newton diagram and their end points are \emph{vertices} of the Newton diagram. If a vertex of the Newton diagram lies on a coordinates axis, then it is said to be an \emph{exterior vertex}; otherwise, it is an \emph{inner vertex}. If the boundary of the convex hull $\Gamma$ of set \eqref{eq-14} contains a ray that does not lie on the coordinate axis, we endow this ray with an \emph{unbounded edge} of the Newton diagram.

Let $\left(a_{ij},b_{ij}\right)$ be the vector coefficient associated with the vertex $\left(i,j\right)$. The \emph{exponent of a vertex} $\left(i,j\right)$ is defined as the number
 \begin{equation*}\label{eq-13}
  \alpha:=
  \begin{cases}
  \frac{b_{ij}}{a_{ij}},\quad\text{if}\quad a_{ij}\neq0,\\
\infty,\quad\text{if}\quad a_{ij}=0. \\
   \end{cases}
\end{equation*}
The \emph{exponent of a bounded edge} $\ell$ of the Newton diagram is a positive rational number $t_2/t_1$  that is equal to the modulus of the tangent of the angle between the edge and the ordinate axis.  The exponent of the edge $\ell$ is denoted by $\alpha_\ell$, and the pair $\mathbf{t}=\left(t_1,t_2\right)$ is called \emph{type} of the edge $\ell$. If the Newton diagram has an unbounded horizontal edge, then we define its exponent as $\infty$ and its type as $\left(0,1\right)$, and if there is an unbounded vertical edge, then we define its exponent as $0$ and its type as $\left(1,0\right)$.

For each inner vertex $V$ of the Newton diagram of the vector field $\mathscr{X}$, there exist $\mathbf{t}=\left(t_1,t_2\right)$ type edge $\ell$ and $\mathbf{s}=\left(s_1,s_2\right)$ type edge $\tilde{\ell}$ such that $\ell$ and $\tilde{\ell}$ are its upper and lower adjacent edges, respectively, i.e., $\alpha_{\ell}=t_2/t_1<s_2/s_1=\alpha_{\tilde{\ell}}$, with $h_{r_{\mathbf{t}}+|\mathbf{t}|}h_{r_{\mathbf{s}}+|\mathbf{s}|}\not\equiv0$. To each inner vertex $V$, we introduce a constant
\begin{equation*}\label{eq-17}
\beta=c_{i_0}\tilde{c}_{j_0},
\end{equation*}
where $i_0=\min\left\{i\geq0\;|\;c_i\neq0\right\}$, $j_0=\min\left\{j\geq0\;|\;\tilde{c}_j\neq0\right\}$, and $c_i$ and $\tilde{c}_j$ are the coefficients of the polynomials $h_{r_{\mathbf{t}}+|\mathbf{t}|}$ and $h_{r_{\mathbf{s}}+|\mathbf{s}|}$, ordered from the highest to the lowest exponent in $x$ and $y$, respectively.

\subsection{Criterion of  monodromic singular point}\label{sub-2}
Let the origin be a singular point of the vector field $\mathscr{X}$.  The origin is \emph{monodromic} if there exists a neighborhood of the origin such that the orbits of $\mathscr{X}$ turn around the origin either $t\rightarrow+\infty$ or $t\rightarrow-\infty$. In  the neighborhood of the origin, we define the following sets
$$W_{\mathbf{t},\mathbf{s}}^{\left(\sigma_1,\sigma_2\right)}=\left\{(x,y)\in \mathbb{R}^2\;|\;\epsilon x^{s_2/s_1}\leq y\leq \frac{1}{\epsilon}x^{t_2/t_1},\;(-1)^{\sigma_1}x\geq0,\;(-1)^{\sigma_2}\epsilon>0\right\},$$
with $\sigma_1,\sigma_2\in\left\{0,1\right\}$.

The following results related to the sufficient conditions of the non-monodromic or monodromic singular point were proved in \cite{MR2819283}.
\begin{proposition}[see \cite{MR2819283}]\label{pr1}
If the Newton diagram of the vector field $\mathscr{X}$ exists an edge with  type $\mathbf{t}$ such that $h_{r_{\mathbf{t}}+|\mathbf{t}|}\equiv0$ and $\mu_{r_{\mathbf{t}}}\not\equiv0$, then the origin is a node (i.e, the origin has  parabolic sectors).
\end{proposition}
\begin{proposition}[see \cite{MR2819283}]\label{pr-7}
Assume that $V$ is an inner vertex of the Newton diagram of the vector field $\mathscr{X}$ with $h_{r_{\mathbf{t}}+|\mathbf{t}|}h_{r_{\mathbf{s}}+|\mathbf{s}|}\not\equiv0$. The region $W_{\mathbf{t},\mathbf{s}}^{\left(0,0\right)}$ is a parabolic (resp. hyperbolic) sector of the origin if and only if $\beta<0$ (resp. $\beta>0$).
\end{proposition}
\begin{theorem}[see \cite{MR2819283}]\label{th-12}
Assume that the Newton diagram of the vector field $\mathscr{X}$ satisfies
the following assumptions:

 \begin{itemize}
   \item  [\emph{(a)}] All its vertices have even coordinates.
   \item[\emph{(b)}]  It has two exterior vertices. Moreover, if $(a,0)$ and $(0,b)$ are the vector coefficients of the exterior vertices, then $ab<0$.
   \item [\emph{(c)}] All its inner vertices $V$ verify $\beta>0$.
   \item [\emph{(d)}]  For each bounded edge, its associated Hamiltonian is non-null and does not have any factor of the form $y^{t_1}-\tilde{a}x^{t_2}$ with $\tilde{a}\in\mathbb{R}\setminus\left\{0\right\}$.
 \end{itemize}
Then the origin of the vector field $\mathscr{X}$ is a monodromic singular point.
\end{theorem}

\section{Proof of Theorems \ref{th-11} and \ref{th2}}\label{se-3}
Our goal of this section is to prove Theorems \ref{th-11} and \ref{th2}.  Firstly, we will give several properties of the Newton diagram for the vector field $b\left(\mathcal{X}\right)$.

\begin{lemma}\label{le-3}
Let $\mathscr{X}$ be a polynomial vector field  and the points $p_i\in \emph{supp}\left(\mathscr{X}\right)$ with $i=1,\ldots,m+1$. The following statements hold.
\begin{itemize}
\item  [\emph{(a)}]  If the points $p_1$, $p_2$ and $p_3$ are collinear  (i.e., $p_1$, $p_2$ and $p_3$ lie on the same line) and $p_2$ is located between $p_1$ and $p_3$, then $p_2$ is not a vertex of the Newton diagram of the vector field $\mathscr{X}$.
\item  [\emph{(b)}] If $p_1,\ldots,p_m$ are the vertices of the polygon in Figure \ref{Fig-4} and $p_{m+1}$  is located in the interior of the polygon, then $p_{m+1}$ is not a vertex of the Newton diagram of the vector field $\mathscr{X}$.
\end{itemize}
\begin{figure}[H]
\centering
\begin{minipage}{0.25\linewidth}
\centering
\psfrag{1}{$p_1$}
\psfrag{2}{$p_2$}
\psfrag{3}{$p_3$}
\psfrag{4}{$p_4$}
\psfrag{5}{$p_m$}
\psfrag{6}{$p_{m+1}$}
\psfrag{7}{$l_m$}
\psfrag{8}{$l_1$}
\psfrag{9}{$l_2$}
\psfrag{10}{$l_3$}
\centerline{\includegraphics[width=1\textwidth]{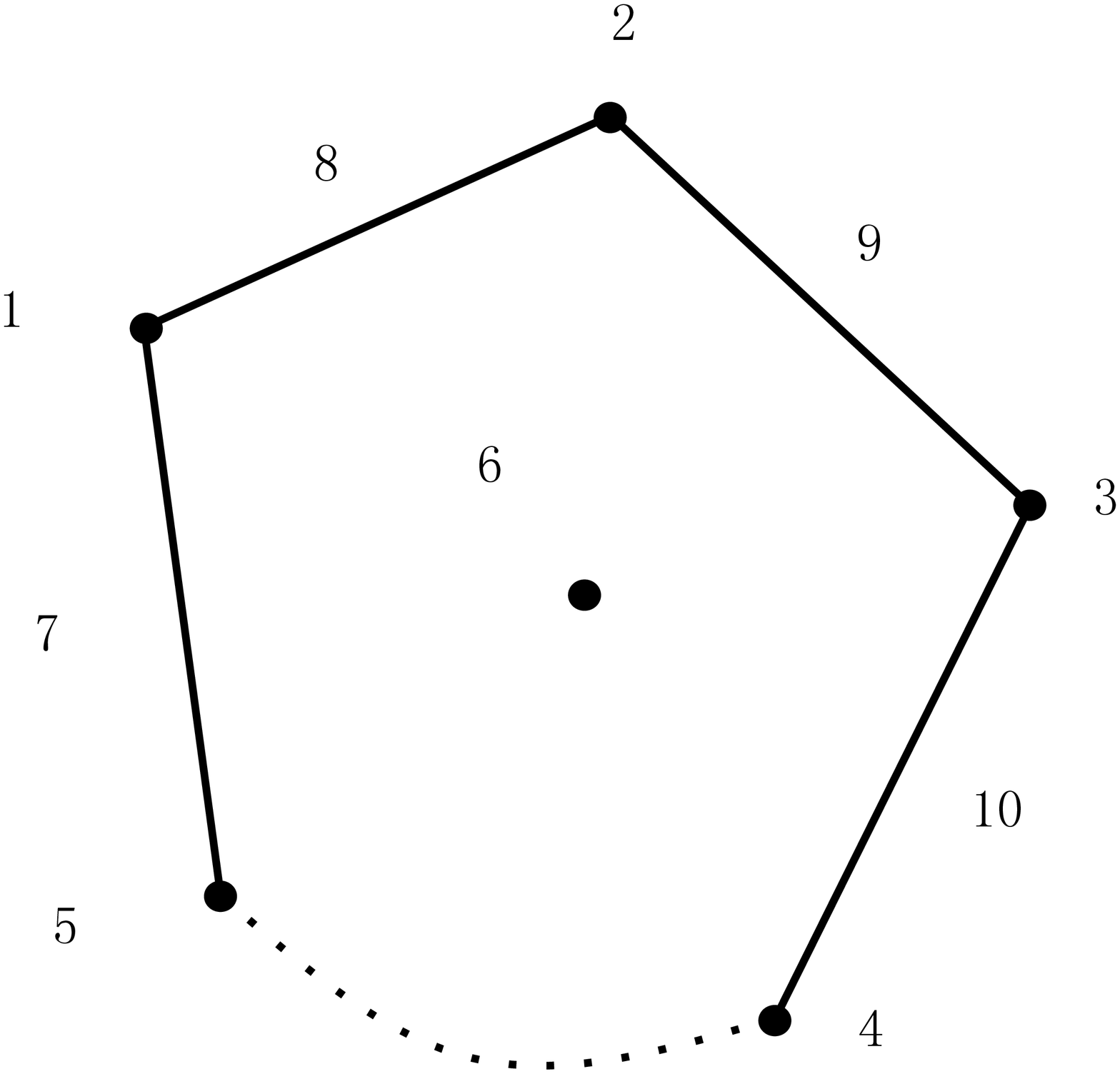}}
\end{minipage}
\caption{Polygon.}\label{Fig-4}
\end{figure}
\end{lemma}
\begin{proof}
$\rm (a)$\; Suppose that $p_2$ is a vertex of the Newton diagram of the vector field $\mathscr{X}$. This means that there exists a line $l$ through $p_2$ such that the convex hull $\Gamma$ of the set
\begin{align}\label{eq14}
	&\bigcup\limits_{\left(i,j\right)\in\text{supp}\left(\mathscr{X}\right)}\left(\left(i,j\right)+\mathbb{R}_+^2\right)
\end{align}
 is located completely in one half-plane separated by the line $l$.  If $p_1\in l$, then $p_3\in l$ and thus $p_2$ is not a vertex, which contradicts with the assumption. If $p_1\not\in l$, then $p_1\not\in\Gamma$ or $p_3\not\in\Gamma$, which is in contradiction with $p_1,p_3\in \Gamma$. Thus the statement $\rm(a)$ is proved.

$\rm(b)$\;Suppose that $p_{m+1}$ is a vertex of the Newton diagram of the vector field $\mathscr{X}$. Then there exists a straight line $l$ through $p_{m+1}$ such that the convex hull $\Gamma$ of the set \eqref{eq14} is located completely in one half-plane separated by the line $l$. Obviously, the line $l$ must intersect one of the edges $\overline{p_ip_{i+1}}$ of the polygon. We get that $p_i\not\in\Gamma$ or $p_{i+1}\not\in\Gamma$, which is in contradiction with $p_i,p_{i+1}\in \Gamma$. The proof is completed.
\end{proof}
\begin{lemma}\label{le-2}
Assume that the polynomial $R\left(x,y\right)\in\mathcal{P}_{k}^{\mathbf{t}}$ and the vector field $\mathscr{X}_k=\left(P_{k+t_1},Q_{k+t_2}\right)\in \mathcal{Q}_k^{\mathbf{t}}$. Let $m_1=\max\left\{\emph{deg}_xP_{k+t_1},1+\emph{deg}_xQ_{k+t_2}\right\}$, $m_2=\max\left\{1+\emph{deg}_yP_{k+t_1},\emph{deg}_yQ_{k+t_2}\right\}$, $m_3=\emph{deg}_xR$ and $m_4=\emph{deg}_yR$. The following statements hold.
  \begin{itemize}
    \item[\emph{(a)}] The support of the vector field $\mathscr{X}_k$ lies on the straight line $t_1x+t_2y=k+|\mathbf{t}|$, and $$V_1=\left(m_1,1+\frac{k+\left(1-m_1\right)t_1}{t_2}\right)\quad and\quad V_2=\left(1+\frac{k+\left(1-m_2\right)t_2}{t_1},m_2\right)$$
        are the vertices of the Newton diagram of $\mathscr{X}_k$.
    \item [\emph{(b)}] The support of the polynomial $R\left(x,y\right)$ lies on the straight line $t_1x+t_2y=k$, and
    $$V_1=\left(m_3,\frac{k-t_1m_3}{t_2}\right)\quad and \quad V_2=\left(\frac{k-t_2m_4}{t_1},m_4\right)$$
    are the vertices of the Newton diagram of $R\left(x,y\right)$.
  \end{itemize}
\end{lemma}

\begin{proof}
  $\rm (a)$\;  Let $yP_{k+t_1}\left(x,y\right)=\sum a_{ij}x^iy^j$ and $xQ_{k+t_2}\left(x,y\right)=\sum b_{ij}x^iy^j$. For arbitrary $\lambda\in \mathbb{R}^+$, we have
$$P_{k+t_1}\left(\lambda^{t_1}x,\lambda^{t_2}y\right)=\sum a_{ij}\lambda^{it_1+\left(j-1\right)t_2}x^iy^{j-1}=\lambda^{k+t_1}P_{k+t_1}\left(x,y\right)=\lambda^{k+t_1}\sum a_{ij}x^iy^{j-1}$$
and
$$Q_{k+t_2}\left(\lambda^{t_1}x,\lambda^{t_2}y\right)=\sum b_{ij}\lambda^{\left(i-1\right)t_1+jt_2}x^{i-1}y^j=\lambda^{k+t_2}Q_{k+t_2}\left(x,y\right)=\lambda^{k+t_2}\sum b_{ij}x^{i-1}y^j.$$
  Thus, $a_{ij}\lambda^{it_1+jt_2}=a_{ij}\lambda^{k+|\mathbf{t}|}$ and  $b_{ij}\lambda^{it_1+jt_2}=b_{ij}\lambda^{k+|\mathbf{t}|}$. To each $\left(i,j\right)\in \text{supp}\left(\mathscr{X}_k\right)$, it satisfies $it_1+jt_2=k+|\mathbf{t}|$ due to the fact that the corresponding vector coefficient $\left(a_{ij},b_{ij}\right)$ is nonzero. This means that the Newton diagram of $\mathscr{X}_k$ is a line segment and its end points
   $$V_1=\left(m_1,1+\frac{k+\left(1-m_1\right)t_1}{t_2}\right)\quad and\quad V_2=\left(1+\frac{k+\left(1-m_2\right)t_2}{t_1},m_2\right)$$
  are the vertices. Hence the conclusion $\rm (a)$ is confirmed.

  The statement $\rm (b)$ can be proved in a similar way.
\end{proof}

From Lemma \ref{le-2}, we have a corollary below.
\begin{corollary}\label{co-2}
Consider the vector field $\mathscr{X}=\sum\mathscr{X}_i$ with $\mathscr{X}_i\in\mathcal{Q}_i^{\mathbf{t}}$. Let $V_{1i}$ and $V_{2i}$ be the vertices of the Newton diagram of $\mathscr{X}_i$. Then the vertices of the Newton diagram of  $\mathscr{X}$ are contained in the set $\bigcup\limits_i\left\{V_{1i},V_{2i}\right\}\subset\emph{supp}\left(\mathscr{X}\right)$.
\end{corollary}
\begin{proof}
  Let $V$ be a vertex of the Newton diagram of  $\mathscr{X}$. Since $V\in\text{supp}\left(\mathscr{X}\right)=\bigcup\limits_{i}\text{supp}\left(\mathscr{X}_i\right)$, there exists $i_0$ such that $V\in \text{supp}\left\{\mathscr{X}_{i_0}\right\}$. By the statement $\rm (a)$ of Lemma \ref{le-2}, the Newton diagram of $\mathscr{X}_{i_0}$ is a line segment with the end points $V_{1i_0}$ and $V_{2i_0}$. Using the statement $\rm (a)$ of Lemma \ref{le-3}, we have $V\in\left\{V_{1i_0},V_{2i_0}\right\}$. The proof is finished.
\end{proof}

\begin{lemma}\label{le-4}
Consider the vector field
  \begin{align}\label{eq-18}
 &\mathscr{Y}=\left(\left(x^2-y^2\right)\partial_y\mathcal{H}-2xy\partial_x\mathcal{H},\left(x^2-y^2\right)\partial_x\mathcal{H}+2xy\partial_y\mathcal{H}\right)
  \end{align}
  with $\mathcal{H}$ a homogeneous polynomial of degree $k$ in the variables $x$, $y$. Let the homogeneous  polynomial $\mathcal{H}=\sum_{i=0}^{m}h_{m-i,k-m+i}x^{m-i}y^{k-m+i}=\sum_{i=0}^n\tilde{h}_{k-n+i,n-i}x^{k-n+i}y^{n-i}$, where $m=\emph{deg}_x\mathcal{H}$, $n=\emph{deg}_y\mathcal{H}$ and $h_{m,k-m}\tilde{h}_{k-n,n}\neq0$. The following statements hold.
  \begin{itemize}
   \item[\emph{(a)}] $\emph{supp}\left(\left(x^2+y^2\right)\mathcal{H}\left(x,y\right)\right)\subset\emph{supp}\left(\mathscr{Y}\right)$.
    \item[\emph{(b)}] The Newton diagrams of the vector field $\mathscr{Y}$ and the polynomial $\left(x^2+y^2\right)\mathcal{H}\left(x,y\right)$ have the same vertices $V_1=\left(m+2,k-m\right)$ and $V_2=\left(k-n,n+2\right)$.
   \item[\emph{(c)}]  For the Newton diagram of the vector field $\mathscr{Y}$, the vector coefficients  associated with the vertices $V_1$ and $V_2$ are
  $\left(\left(k\!-\!m\right)h_{m,k-m},\left(2k\!-\!m\right)h_{m,k-m}\right)$ and $\left(\left(n\!-\!2k\right)\tilde{h}_{k-n,n},\left(n\!-\!k\right)\tilde{h}_{k-n,n}\right)$, respectively.
  \end{itemize}
\end{lemma}
\begin{proof}
$\rm(a)$\; Since $x\left(\left(x^2-y^2\right)\partial_x\mathcal{H}+2xy\partial_y\mathcal{H}\right)
-y\left(\left(x^2-y^2\right)\partial_y\mathcal{H}-2xy\partial_x\mathcal{H}\right)=k\left(x^2+y^2\right)\mathcal{H}$,
one has the statement $\rm(a)$ holds.

Next we prove the statements $\rm(b)$ and $\rm(c)$.

For $m=0$, the conclusion holds obviously.  For $m\geq1$,
we distinguish between the cases $m<k$ and $m=k$.

If $m<k$, then
$$\partial_x\mathcal{H}=\sum_{i=0}^{m-1}\left(m-i\right)h_{m-i,k-m+i}x^{m-i-1}y^{k-m+i}$$
and
  $$\partial_y\mathcal{H}=\sum_{i=0}^{m}\left(k-m+i\right)h_{m-i,k-m+i}x^{m-i}y^{k-m+i-1}.$$ We have
  \begin{equation*}\label{eq-19}
  y\left(\left(x^2-y^2\right)\partial_y\mathcal{H}-2xy \partial_x\mathcal{H}\right)=\left(k-m\right)h_{m,k-m}x^{m+2}y^{k-m}+\varphi\left(x,y\right)\\[2ex]
  \end{equation*}
  and
  \begin{equation*}\label{eq-20}
x\left(\left(x^2-y^2\right)\partial_x\mathcal{H}+2xy \partial_y\mathcal{H}\right)=\left(2k-m\right)h_{m,k-m}x^{m+2}y^{k-m}+\psi\left(x,y\right),\\[2ex]
  \end{equation*}
where $\text{deg}_x\varphi<m+2$ and $\text{deg}_x\psi<m+2$. So,
  $$\text{deg}_x\left(y\left(\left(x^2-y^2\right)\partial_y\mathcal{H}-2xy\partial_x\mathcal{H}\right)\right)=
  \text{deg}_x\left(x\left(\left(x^2-y^2\right)\partial_x\mathcal{H}+2xy\partial_y\mathcal{H}\right)\right)=m+2.$$

If $m=k$, then  $\partial_x\mathcal{H}=\sum_{i=0}^{k-1}\left(k-i\right)h_{k-i,i}x^{k-i-1}y^{i}$ and
  $\partial_y\mathcal{H}=\sum_{i=1}^kih_{k-i,i}x^{k-i}y^{i-1}$. We get
\begin{equation*}\label{eq-21}
\text{deg}_x\left(y\left(\left(x^2-y^2\right)\partial_y\mathcal{H}-2xy\partial_x\mathcal{H}\right)\right)<k+2\\[2ex]
\end{equation*}
and
 \begin{equation*}\label{eq-22}
x\left(\left(x^2-y^2\right)\partial_x\mathcal{H}+2xy\partial_y\mathcal{H}\right)=kh_{k,0}x^{k+2}+\widetilde{\psi}\left(x,y\right)\\[2ex]
 \end{equation*}
with $\text{deg}_x\widetilde{\psi}<k+2.$

On the other hand, we have $\mathscr{Y}\in\mathcal{Q}_k^{\mathbf{t}}$ and $\mathcal{H}\in\mathcal{P}_{k}^{\mathbf{t}}$ with $\mathbf{t}=\left(1,1\right)$. Thus, by Lemma \ref{le-2}, the Newton diagrams of the vector field $\mathscr{Y}$ and the polynomial $\left(x^2+y^2\right)\mathcal{H}\left(x,y\right)$ have the vertex $V_1=\left(m+2,k-m\right)$.  For the Newton diagram of $\mathscr{Y}$, the vector coefficient of the vertex $V_1$ is  $\left(\left(k-m\right)h_{m,k-m},\left(2k-m\right)h_{m,k-m}\right)$.
\iffalse
{\bf{For $1\leq m\leq k$,
$$\partial_x\mathcal{H}=\sum_{i=0}^{m-1}\left(m-i\right)h_{m-i,k-m+i}x^{m-i-1}y^{k-m+i}$$
and
  $$\partial_y\mathcal{H}=\sum_{i=0}^{m}\left(k-m+i\right)h_{m-i,k-m+i}x^{m-i}y^{k-m+i-1}.$$ We have
  \begin{align}\label{eq-19}
  &y\left(\left(x^2-y^2\right)\partial_y\mathcal{H}-2xy \partial_x\mathcal{H}\right)=\left(k-m\right)h_{m,k-m}x^{m+2}y^{k-m}+\varphi\left(x,y\right)
  \end{align}
  and
  \begin{align}\label{eq-20}
&x\left(\left(x^2-y^2\right)\partial_x\mathcal{H}+2xy \partial_y\mathcal{H}\right)=\left(2k-m\right)h_{m,k-m}x^{m+2}y^{k-m}+\psi\left(x,y\right),
  \end{align}
where $\text{deg}_x\varphi<m+2$ and $\text{deg}_x\psi<m+2$. So,
$$\text{deg}_x\left(y\left(\left(x^2-y^2\right)\partial_y\mathcal{H}-2xy\partial_x\mathcal{H}\right)\right)\leq m+2,$$
$$\text{deg}_x\left(x\left(\left(x^2-y^2\right)\partial_x\mathcal{H}+2xy\partial_y\mathcal{H}\right)\right)=m+2.$$

On the other hand, we have $\mathscr{Y}\in\mathcal{Q}_k^{\mathbf{t}}$ and $\mathcal{H}\in\mathcal{P}_{k}^{\mathbf{t}}$ with $\mathbf{t}=\left(1,1\right)$. Thus, by Lemma \ref{le-2}, the Newton diagrams of the vector field $\mathscr{Y}$ and the polynomial $\left(x^2+y^2\right)\mathcal{H}\left(x,y\right)$ have the vertex $V_1=\left(m+2,k-m\right)$.  For the Newton diagram of $\mathscr{Y}$, the vector coefficient of the vertex $V_1$ is  $\left(\left(k-m\right)h_{m,k-m},\left(2k-m\right)h_{m,k-m}\right)$.}}
\fi

The proof of the vertex $V_2$ is analogous to the proof of the vertex $V_1$, and thus we omit it here.
\end{proof}

Let $f\left(x,y\right)=\sum\nolimits_{i=1}^{n}f_i\left(x,y\right)$, $g\left(x,y\right)=\sum\nolimits_{j=1}^{m}g_j\left(x,y\right)$ and $d=\max\{n,m\}$, where $f_i\left(x,y\right)$ and $g_j\left(x,y\right)$ are homogeneous polynomials of degree $i$ and $j$, respectively. From the equation \eqref{eq-5}, the expression of $b\left(\mathcal{X}\right)$  is given by
\begin{align}\label{eq-6}
	\begin{cases}
		\dot{u}=&\sum\limits_{i=1}^d\sum\limits_{j=1}^d\left(u^2+v^2\right)^{2d-i-j}\left[\left(u^2-v^2\right)\left(f_i\left(u,v\right)f_{jy}\left(u,v\right)
		+g_i\left(u,v\right)g_{jy}\left(u,v\right)\right)\right.\\
		&\left.-2uv\left(f_i\left(u,v\right)f_{jx}\left(u,v\right)+g_i\left(u,v\right)g_{jx}\left(u,v\right)\right)\right],\\
		\specialrule{0em}{3pt}{3pt}
		\dot{v}=&\sum\limits_{i=1}^d\sum\limits_{j=1}^d\left(u^2+v^2\right)^{2d-i-j}\left[\left(u^2-v^2\right)\left(f_i\left(u,v\right)f_{jx}\left(u,v\right)
		+g_i\left(u,v\right)g_{jx}\left(u,v\right)\right)\right.\\
		&\left.+2uv\left(f_i\left(u,v\right)f_{jy}\left(u,v\right)+g_i\left(u,v\right)g_{jy}\left(u,v\right)\right)\right].
	\end{cases}
\end{align}
Furthermore, the vector field $b\left(\mathcal{X}\right)$ can be rewritten as  the sum of its  homogeneous components
\begin{align}\label{eq-24}
	&b\left(\mathcal{X}\right)=\frac{1}{2}\sum_{i=1}^d\mathcal{F}_{4d+1-2i}+\sum_{1\leq i<j\leq d}\mathcal{F}_{4d+1-i-j},
\end{align}
where
\begin{small}
	\begin{align}\label{eq-25}
		&\mathcal{F}_{4d+1-i-j}=\left(
		\begin{array}{l}
			\left(u^2+v^2\right)^{2d-i-j}\left[\left(u^2-v^2\right)\partial_v\left(f_if_j+g_ig_j\right)
			-2uv\partial_u\left(f_if_j+g_ig_j\right)\right]\\
			\left(u^2+v^2\right)^{2d-i-j}\left[\left(u^2-v^2\right)\partial_u\left(f_if_j+g_ig_j\right)
			+2uv\partial_v\left(f_if_j+g_ig_j\right)\right]\\
		\end{array}
		\right)^T\in\mathcal{Q}_{4d+1-i-j}^{\left(1,1\right)}
	\end{align}
\end{small}for $i,j=1,\ldots,d$. Note that $f_{jx}(u,v)=f_{ju}(u,v)$, $f_{jy}(u,v)=f_{jv}(u,v)$, $g_{jx}(u,v)=g_{ju}(u,v)$ and $g_{jy}(u,v)=g_{jv}(u,v)$.
\begin{proposition}\label{pr-6}
Consider the vector field
\begin{align}\label{eq-23}
&\widetilde{\mathscr{Y}}=\frac{1}{2}\sum_{i=1}^d\mathcal{F}_{4d+1-2i},
\end{align}
where $\mathcal{F}_{4d+1-2i}$ are given by the equation \eqref{eq-25} with $i=1,\ldots,d$. Then the Newton diagrams of $\widetilde{\mathscr{Y}}$ and $b\left(\mathcal{X}\right)$ given in \eqref{eq-24} have the same vertices.
\end{proposition}

\begin{proof}
 It is sufficient to prove that the vertices of the Newton diagram of $\mathcal{F}_{4d+1-i-j}$  are not the vertices of the Newton diagram of $b\left(\mathcal{X}\right)$ for $i\neq j$.

From Lemma \ref{le-4}, the Newton diagrams of $\mathcal{F}_{4d+1-i-j}$ and $\left(u^2+v^2\right)^{2d+1-i-j}\left(f_if_j+g_ig_j\right)$ have the same vertices for $i,j=1,\ldots,d$.

Let $m_i=\max\left\{\text{deg}_xf_i,\text{deg}_xg_i\right\}$, $m_j=\max\left\{\text{deg}_xf_j,\text{deg}_xg_j\right\}$, $n_i=\max\left\{\text{deg}_yf_i,\text{deg}_yg_i\right\}$ and $n_j=\max\left\{\text{deg}_yf_j,\text{deg}_yg_j\right\}$. For $i\neq j$, the following relations hold:
$$\text{deg}_x\left(f_i^2+g_i^2\right)=2m_i,\;m_{i+j}:=\text{deg}_x\left(f_if_j+g_ig_j\right)\leq m_i+m_j,\;\text{deg}_x\left(f_j^2+g_j^2\right)=2m_j;$$ $$\text{deg}_y\left(f_i^2+g_i^2\right)=2n_i,\;n_{i+j}:=\text{deg}_y\left(f_if_j+g_ig_j\right)\leq n_i+n_j,\;\text{deg}_y\left(f_j^2+g_j^2\right)=2n_j.$$
Applying Lemmas \ref{le-2} and \ref{le-4} again,  we obtain that the vertices of the Newton diagram of $\mathcal{F}_{4d+1-2i}$ are
\begin{small}
\begin{align}\label{V12even}
&V_{1,2i}=\left(2\left(2d+1-2i+m_i\right),2\left(i-m_i\right)\right)\quad \text{and} \quad V_{2,2i}=\left(2\left(i-n_i\right),2\left(2d+1-2i+n_i\right)\right);
\end{align}
\end{small}the vertices of the Newton diagram of $\mathcal{F}_{4d+1-i-j}$ are
\begin{small}
\begin{align*}
&V_{1,i+j}=\left(2\left(2d+1-i-j\right)\!+\!m_{i+j},i+j\!-\!m_{i+j}\right)\quad \text{and} \quad V_{2,i+j}=\left(i\!+\!j\!-\!n_{i+j},2\left(2d+1-i-j\right)\!+\!n_{i+j}\right);
\end{align*}
\end{small}and the vertices of the Newton diagram of $\mathcal{F}_{4d+1-2j}$ are
\begin{small}
\begin{align*}
&V_{1,2j}=\left(2\left(2d+1-2j+m_j\right),2\left(j-m_j\right)\right)\quad \text{and} \quad V_{2,2j}=\left(2\left(j-n_j\right),2\left(2d+1-2j+n_j\right)\right).
\end{align*}
\end{small}Without loss of generality, we can assume $i<j$. From Lemma \ref{le-2}, it follows that $V_{1,2i}$ and $V_{2,2i}$ lie on the straight line $l_1: x+y=2\left(2d+1-i\right)$, $V_{1,i+j}$ and $V_{2,i+j}$ lie on the straight line $l_2: x+y=2\left(2d+1\right)-i-j$, and $V_{1,2j}$ and $V_{2,2j}$ lie on the straight line $l_3: x+y=2\left(2d+1-j\right)$. Consequently, the vertices $V_{1,2i}$, $V_{2,2i}$, $V_{1,2j}$ and $V_{2,2j}$ have three possible configurations  given in Figure \ref{Fig-5}, where $A=\left(i+j-n_i-n_j,2\left(2d+1-i-j\right)+n_i+n_j\right)$ and $B=\left(2\left(2d+1-i-j\right)+m_i+m_j,i+j-m_i-m_j\right)$.

\begin{figure}[H]
\scriptsize
\centering
\begin{minipage}{0.225\linewidth}
\centering
\psfrag{0}{$x$}
\psfrag{1}{$y$}
\psfrag{2}{$l_3$}
\psfrag{3}{$l_2$}
\psfrag{4}{$l_1$}
\psfrag{5}{$V_j$}
\psfrag{6}{$C$}
\psfrag{7}{$V_i$}
\centerline{\includegraphics[width=1\textwidth]{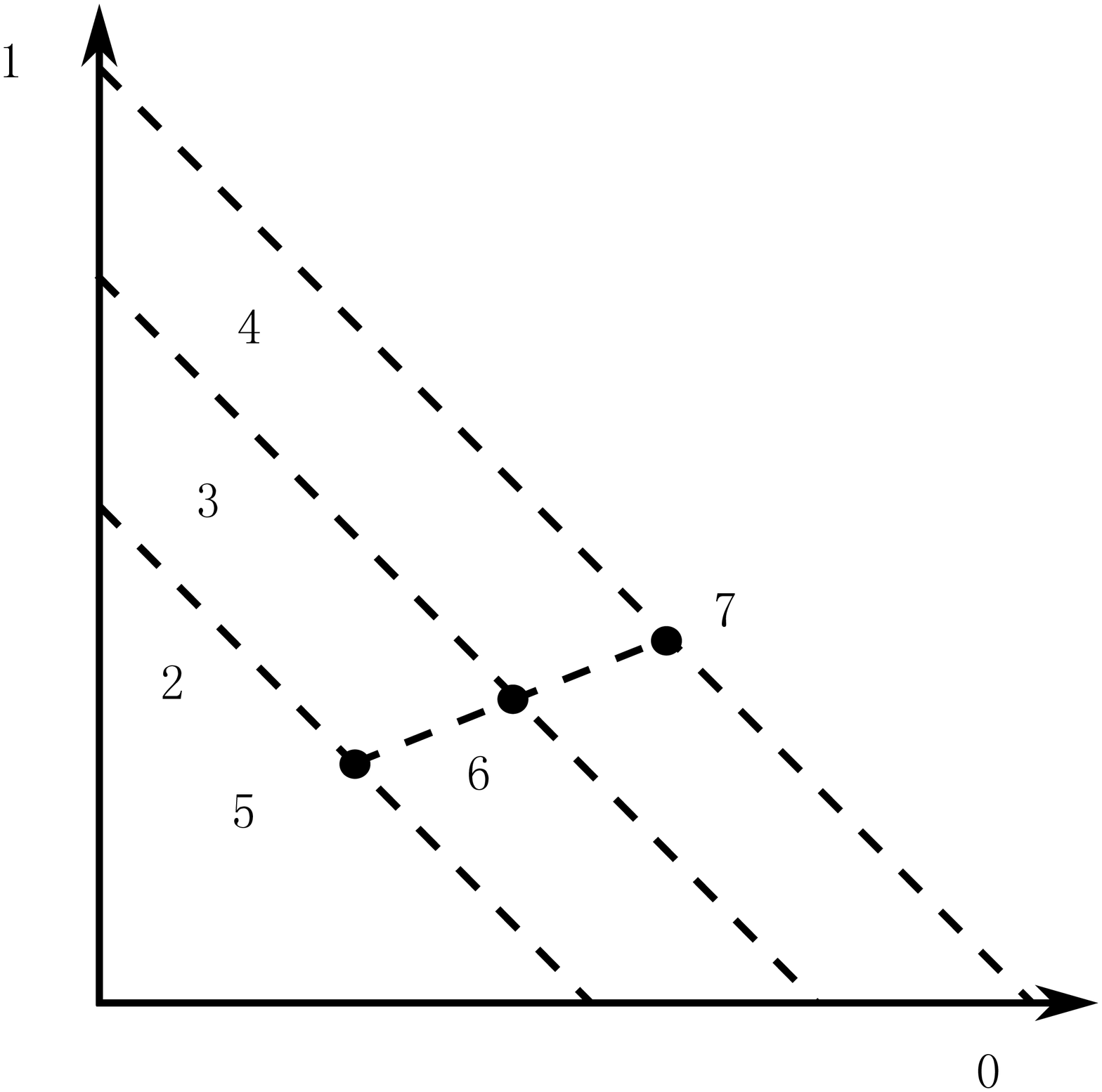}}
\centerline{(i)\;$V_i=V_{1,2i}=V_{2,2i}$\;and\;$V_j=V_{1,2j}=V_{2,2j}$}
\end{minipage}
\qquad
\qquad
\qquad
\begin{minipage}{0.225\linewidth}
\centering
\psfrag{0}{$x$}
\psfrag{1}{$y$}
\psfrag{2}{$l_3$}
\psfrag{3}{$l_2$}
\psfrag{4}{$l_1$}
\psfrag{5}{$V_i$}
\psfrag{6}{$A$}
\psfrag{7}{$V_{2,2j}$}
\psfrag{8}{$V_{1,2j}$}
\psfrag{9}{$B$}
\centerline{\includegraphics[width=1\textwidth]{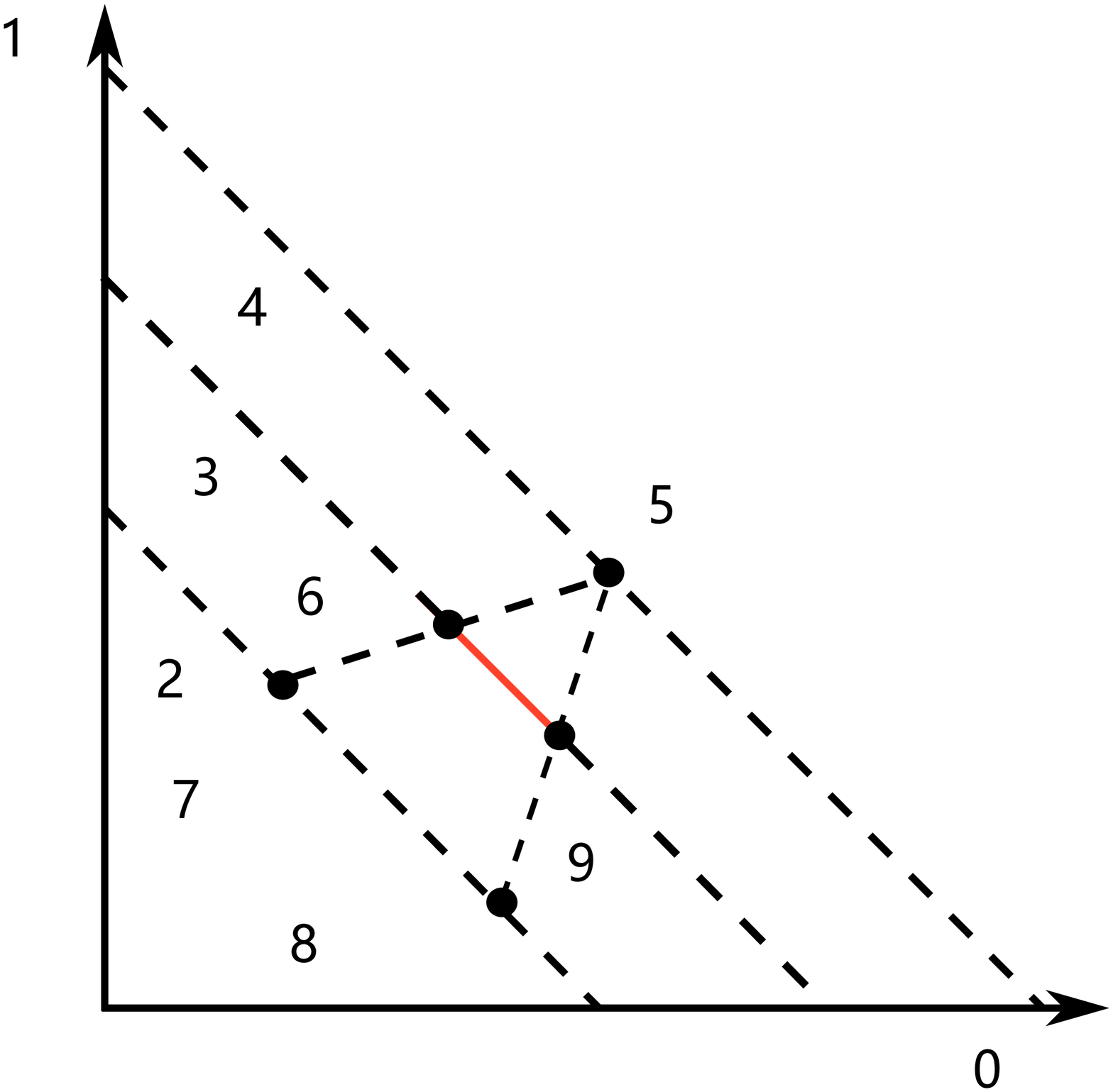}}
\centerline{(ii)\;$V_i=V_{1,2i}=V_{2,2i}$\;and\;$V_{1,2j}\neq V_{2,2j}$}
\end{minipage}
\qquad
\qquad
\begin{minipage}{0.225\linewidth}
\centering
\psfrag{0}{$x$}
\psfrag{1}{$y$}
\psfrag{2}{$l_3$}
\psfrag{3}{$l_2$}
\psfrag{4}{$l_1$}
\psfrag{5}{$V_{2,2j}$}
\psfrag{6}{$A$}
\psfrag{7}{$V_{2,2i}$}
\psfrag{8}{$V_{1,2i}$}
\psfrag{9}{$B$}
\psfrag{10}{$V_{1,2j}$}
\centerline{\includegraphics[width=1\textwidth]{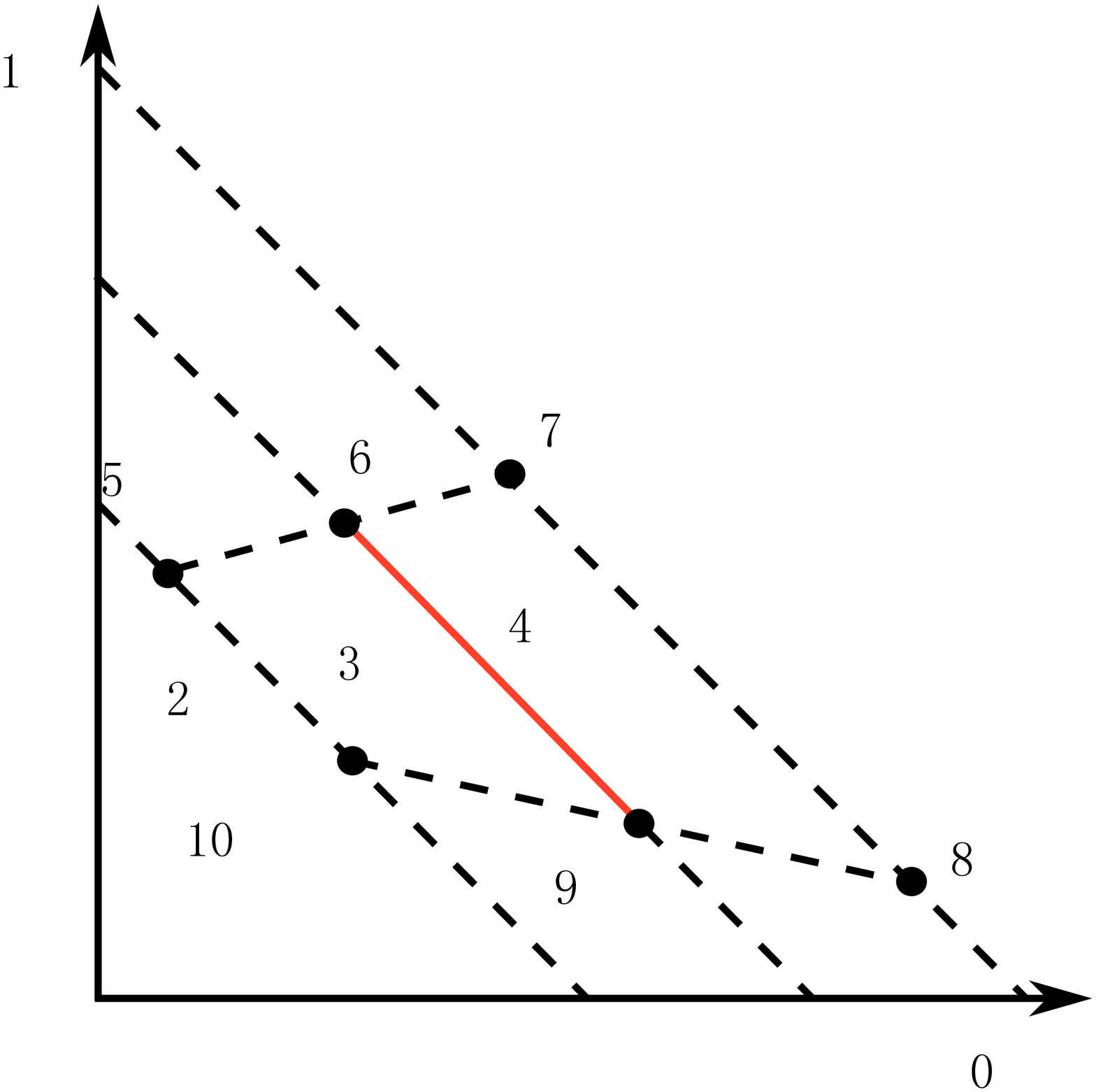}}
\centerline{(iii)\;$V_{1,2i}\neq V_{2,2i}$\;and\;$V_{1,2j}\neq V_{2,2j}$}
\end{minipage}
\caption{Possible configurations of the vertices.}\label{Fig-5}
\end{figure}
{\bf Configuration (i):} If $V_i=V_{1,2i}=V_{2,2i}$ and $V_j=V_{1,2j}=V_{2,2j}$, then we get that $C=V_{1,i+j}=V_{2,i+j}$ is located at the  line segment determined by $V_i$ and $V_j$, and it is not the vertex of the Newton diagram of $b\left(\mathcal{X}\right)$ from Lemma \ref{le-3}.

{\bf Configuration (ii):} If $V_i=V_{1,2i}=V_{2,2i}$ and $V_{1,2j}\neq V_{2,2j}$, then $V_{1,i+j}$ and $V_{2,i+j}$  belong to the line segment $AB$, where $A$ and $B$ belong to the line segments $V_iV_{2,2j}$ and $V_iV_{1,2j}$, respectively.  Thus, by Lemma \ref{le-3}, $V_{1,i+j}$ and $V_{2,i+j}$ are not the vertices of the Newton diagram of $b\left(\mathcal{X}\right)$. Analogously, one can prove the case when $V_{1,2i}\neq V_{2,2i}$ and $V_{1,2j}=V_{2,2j}$.

{\bf Configuration (iii):} By the same reasons as {\bf Configuration (ii)}, $V_{1,i+j}$ and $V_{2,i+j}$ are not the vertices of the Newton diagram of $b\left(\mathcal{X}\right)$.

The proof of the proposition is completed.
\end{proof}

The next  result can be found in \cite{tian2021necessary}, which describes a property of the origin of the vector field $b\left(\mathcal{X}\right)$.
\begin{proposition}[see \cite{tian2021necessary}]\label{pr-5}
	Let $F=\left(f,g\right):\mathbb{R}^2\rightarrow \mathbb{R}^2$ be a polynomial map such that $F\left(0,0\right)=\left(0,0\right)$ and  $\det DF\left(x,y\right)\neq0$ for all  $\left(x,y\right)\in\mathbb{R}^2$. Then the origin of the vector field $b\left(\mathcal{X}\right)$ has no parabolic sectors.
\end{proposition}

The properties of the Newton diagram for the vector field $b\left(\mathcal{X}\right)$ are given as follows.
\begin{theorem}\label{th-9}
Let $\mathcal{N}\left(b\left(\mathcal{X}\right)\right)$ be the Newton diagram of system \eqref{eq-6}. The following statements hold.
\begin{itemize}
\item [\emph{(a)}] For each bounded edge, its associated Hamiltonian is non-null.
  \item [\emph{(b)}] The Newton diagram $\mathcal{N}\left(b\left(\mathcal{X}\right)\right)$ has two exterior vertices. Moreover, if $\left(\tilde{a},0\right)$ and $\left(0,\tilde{b}\right)$ are the vector coefficients of the two exterior vertices of $\mathcal{N}\left(b\left(\mathcal{X}\right)\right)$, then $\tilde{a}  \tilde{b}<0$.
  \item [\emph{(c)}]All the inner vertices of $\mathcal{N}\left(b\left(\mathcal{X}\right)\right)$ satisfy $\beta>0$.
  \item [\emph{(d)}] All the vertices of $\mathcal{N}\left(b\left(\mathcal{X}\right)\right)$ have even coordinates.
\end{itemize}
\end{theorem}

\begin{proof}
$\rm(a)$ Suppose that there exists an edge whose associated Hamiltonian $h_{r_{\mathbf{t}}+|\mathbf{t}|}\equiv0$. Then, by Proposition \ref{pr1},  the origin of $b\left(\mathcal{X}\right)$ is a node. This contradicts with Proposition \ref{pr-5}.

$\rm(b)$ From Proposition \ref{pr-6}, one has that the Newton diagrams of $\widetilde{\mathscr{Y}}$ and $b\left(\mathcal{X}\right)$ have the same exterior vertices if $\mathcal{N}\left(b\left(\mathcal{X}\right)\right)$ has exterior vertices. For the vector field $\widetilde{\mathscr{Y}}$, we have
\begin{equation}\begin{split}\label{eq10}
v\cdot\dot{u}\big|_{u=0}&=-v^2\sum_{i=1}^dv^{4d-4i}\left(vf_i\left(0,v\right)f_{iy}\left(0,v\right)+vg_i\left(0,v\right)g_{iy}\left(0,v\right)\right)\\
&=-v^2\sum_{i=1}^div^{4d-4i}\left(f_i^2\left(0,v\right)+g_i^2\left(0,v\right)\right)
\end{split}\end{equation}
and
\begin{equation}\begin{split}\label{eq11}
u\cdot\dot{v}\big|_{v=0}&=u^2\sum_{i=1}^du^{4d-4i}\left(uf_i\left(u,0\right)f_{ix}\left(u,0\right)+ug_i\left(u,0\right)g_{ix}\left(u,0\right)\right)\\
&=u^2\sum_{i=1}^diu^{4d-4i}\left(f_i^2\left(u,0\right)+g_i^2\left(u,0\right)\right).
\end{split}\end{equation}
It is obvious that $v\cdot\dot{u}\big|_{u=0}\not\equiv0$ and $u\cdot\dot{v}\big|_{v=0}\not\equiv0$. So $\mathcal{N}\left(b\left(\mathcal{X}\right)\right)$ has two exterior vertices. The coefficients of the lowest term of $v\cdot\dot{u}\big|_{u=0}$ in $v$ and the lowest term of $u\cdot\dot{v}\big|_{v=0}$ in $u$ are
negative and positive, respectively. The statement $\rm(b)$ is confirmed.

$\rm(c)$ By Proposition \ref{pr-5}, we have that the origin of the vector field $b\left(\mathcal{X}\right)$ has no parabolic sectors. From Proposition \ref{pr-7}, it follows that $\beta>0$.

$\rm(d)$ Applying Lemma \ref{le-4}, the Newton diagrams of $\mathcal{F}_{4d+1-2i}$ and $\left(u^2+v^2\right)^{2d+1-2i}\left(f_i^2+g_i^2\right)$ have the same vertices for $i=1,\ldots,d$. Thus, all the vertices of $\mathcal{N}\left(\mathcal{F}_{4d+1-2i}\right)$ have even coordinates, see \eqref{V12even}. From Proposition \ref{pr-6}, the statement $\rm(d)$ holds.

This ends the proof.
\end{proof}

Tian and Zhao \cite{tian2021necessary}  proved the following equivalent results, which reveal a relationship between the global property of a polynomial map $F$ and the local dynamical property of $b\left(\mathcal{X}\right)$.
\begin{theorem}[see \cite{tian2021necessary}]\label{th-3}
	Let $F=\left(f,g\right):\mathbb{R}^2\rightarrow \mathbb{R}^2$ be a polynomial map such that $F\left(0,0\right)=\left(0,0\right)$ and  $\det DF\left(x,y\right)\neq0$ for all  $\left(x,y\right)\in\mathbb{R}^2$. Then the following statements are equivalent.
	\begin{itemize}
		\item [\emph{(a)}] $F$ is a global diffeomorphism of the plane onto itself.
		\item [\emph{(b)}]The origin of the vector field $b\left(\mathcal{X}\right)$ is a monodromic singular point.
	\end{itemize}
\end{theorem}

\begin{proof}[\bf Proof of Theorem \ref{th-11}]
Using Theorems \ref{th-12}  and \ref{th-9}, we have that the origin of $b\left(\mathcal{X}\right)$ is a monodromic singular point. So this theorem follows immediately from Theorem \ref{th-3}.
\end{proof}

\begin{proof}[\bf Proof of Theorem \ref{th2}]
Let $f\left(x,y\right)=\sum\nolimits_{i=1}^{n}f_i\left(x,y\right)$, $g\left(x,y\right)=\sum\nolimits_{j=1}^{m}g_j\left(x,y\right)$ and $d=\max\{n,m\}$, where $f_i\left(x,y\right)$ and $g_j\left(x,y\right)$ are homogeneous polynomials of degree $i$ and $j$, respectively. The Hamiltonian of the vector field \eqref{eq-1} is
\begin{align}\label{eq8}
&H\left(x,y\right)=\frac{f^2\left(x,y\right)+g^2\left(x,y\right)}{2}=\frac{1}{2}\sum_{l=2}^{2d}H_l\left(x,y\right),
\end{align}
where $H_l=f_if_j+g_ig_j$ is a homogeneous polynomial of degree $l$ and $i+j=l$ for $i,j=1,\ldots,d$. Obviously, $\deg H=2d$ and $H_{2d,x}^2+H_{2d,y}^2\not\equiv0$.  Our proof  can be divided into three cases.

 {\bf Case 1: $\deg H_x=\deg H_y$.} For this case, the higher homogeneous terms of the polynomials $ff_x+gg_x$ and $ff_y+gg_y$ are
 $$\frac{1}{2}H_{2d,x}\left(x,y\right)=f_d\left(x,y\right)f_{d,x}\left(x,y\right)+g_d\left(x,y\right)g_{d,x}\left(x,y\right)\not\equiv0$$
 and
 $$\frac{1}{2}H_{2d,y}\left(x,y\right)=f_d\left(x,y\right)f_{d,y}\left(x,y\right)+g_d\left(x,y\right)g_{d,y}\left(x,y\right)\not\equiv0,$$
 respectively. Note that $H_{2d}\left(x,0\right)=f_d^2\left(x,0\right)+g_d^2\left(x,0\right)\not\equiv0$ and $H_{2d}\left(0,y\right)=f_d^2\left(0,y\right)+g_d^2\left(0,y\right)\not\equiv0$. If $H_{2d}\left(x,0\right)\equiv0$, then $y|f_d\left(x,y\right)$ and $y|g_d\left(x,y\right)$. So, $y|H_{2d,x}\left(x,y\right)$ and $y|H_{2d,y}\left(x,y\right)$.  This contradicts with the fact that $H_{2d,x}\left(x,y\right)$ and $H_{2d,y}\left(x,y\right)$ do not have real linear factors in common.  Analogously, $H_{2d}\left(0,y\right)\not\equiv0$.
Thereby, $f_d^2\left(x,0\right)+g_d^2\left(x,0\right)=ax^{2d}$ and
 $f_d^2\left(0,y\right)+g_d^2\left(0,y\right)=by^{2d}$ with $a>0,b>0$.

 By  Proposition \ref{pr-6}, the Newton diagrams of
 \begin{align*}
 	&\widetilde{\mathscr{Y}}=\frac{1}{2}\sum_{i=1}^d\mathcal{F}_{4d+1-2i}
 \end{align*}
  and $b\left(\mathcal{X}\right)$ have the same vertices, where
  \begin{small}
  	\begin{align*}
  		&\mathcal{F}_{4d+1-2i}=\left(
  		\begin{array}{l}
  			\left(u^2+v^2\right)^{2d-2i}\left[\left(u^2-v^2\right)\partial_v\left(f_i^2+g_i^2\right)
  			-2uv\partial_u\left(f_i^2+g_i^2\right)\right]\\
  			\left(u^2+v^2\right)^{2d-2i}\left[\left(u^2-v^2\right)\partial_u\left(f_i^2+g_i^2\right)
  			+2uv\partial_v\left(f_i^2+g_i^2\right)\right]\\
  		\end{array}
  		\right)^T
  	\end{align*}
  \end{small}for $i=1,\ldots,d$. For the vector field $\widetilde{\mathscr{Y}}$, we have $$v\cdot\dot{u}\big|_{u=0}=-bdv^{2d+2}+o\left(v^{2d+2}\right)\quad\text{and}\quad u\cdot\dot{v}\big|_{v=0}=adu^{2d+2}+o\left(u^{2d+2}\right)$$ (see the equations \eqref{eq10} and \eqref{eq11}). Thus, the exterior vertices  of $\mathcal{N}\left(b\left(\mathcal{X}\right)\right)$ are $V_0=\left(0,2d+2\right)$ and $V_1=\left(2d+2,0\right)$.  It follows from Lemma \ref{le-4} that the Newton diagrams of $\mathcal{F}_{4d+1-2i}$ and $\left(u^2+v^2\right)^{2d+1-2i}\left(f_i^2+g_i^2\right)$ have the same vertices for $i=1,\ldots,d$. Therefore, by the statement $\rm (b)$ of Lemma \ref{le-2},  the vertices of the Newton diagram of $\mathcal{F}_{4d+1-2i}$ lie on the straight line $l_{4d+1-2i}: x+y=4d+2-2i$ for $i=1,\ldots,d$. So, the configuration of the support for the vector field $b\left(\mathcal{X}\right)$ is  described in (i) of Figure \ref{fig1}.   The Newton diagram of
  $b\left(\mathcal{X}\right)$ consists of two exterior vertices $V_0=\left(0,2d+2\right)$ and $V_1=\left(2d+2,0\right)$, and a unique bounded edge of type $\left(1,1\right)$, see (ii) of Figure \ref{fig1}.
  \begin{figure}[H]
  		\centering
  		\begin{minipage}{0.35\linewidth}
  			\centering
  			\psfrag{0}{$V_0$}
  			\psfrag{1}{$V_1$}
  			\psfrag{2}{$l_{2d+1}$}
  			\psfrag{3}{$l_{4d-1}$}
  			\centerline{\includegraphics[width=1\textwidth]{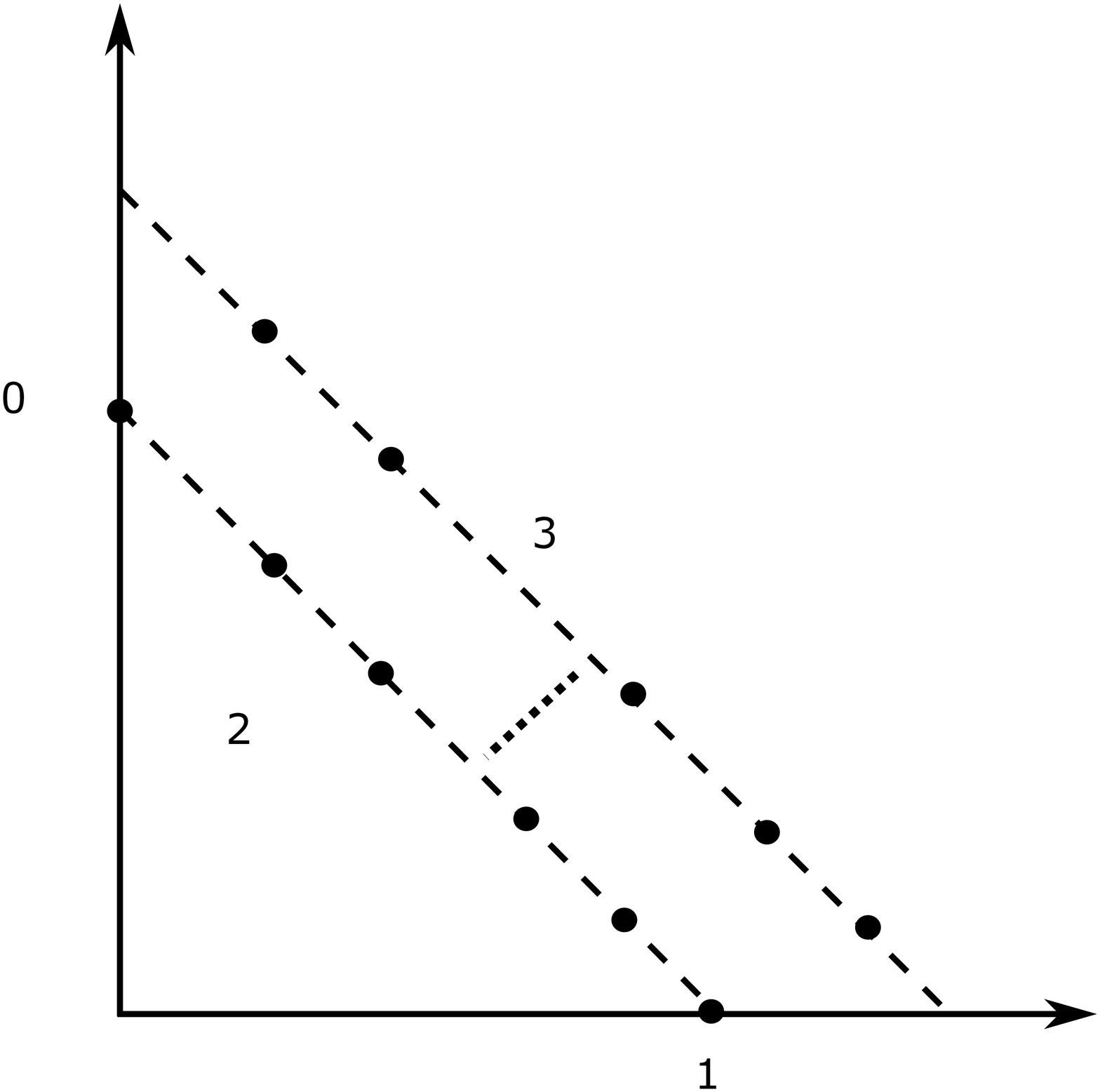}}
  			\vspace{1ex}
  			\centerline{(i)\;Configuration of the support of $b\left(\mathcal{X}\right)$.}
  		\end{minipage}
  		\begin{minipage}{0.15\linewidth}
  			\centering
  			\centerline{\includegraphics[width=0.4\textwidth]{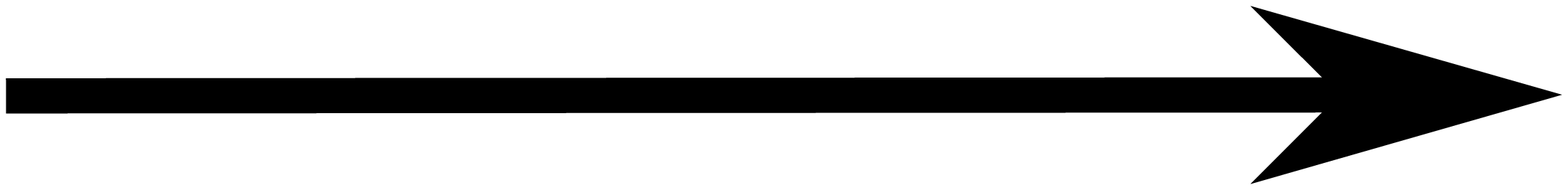}}
  		\end{minipage}
  		\begin{minipage}{0.35\linewidth}
  			\centering
  			\psfrag{0}{$V_0$}
  		\psfrag{1}{$V_1$}
  		\psfrag{2}{$h_{2\left(d+1\right)}^{\left(1,1\right)}=\frac{d}{2\left(d+1\right)}\left(u^2+v^2\right)\left(f_d^2+g_d^2\right)$}
  			\centerline{\includegraphics[width=1\textwidth]{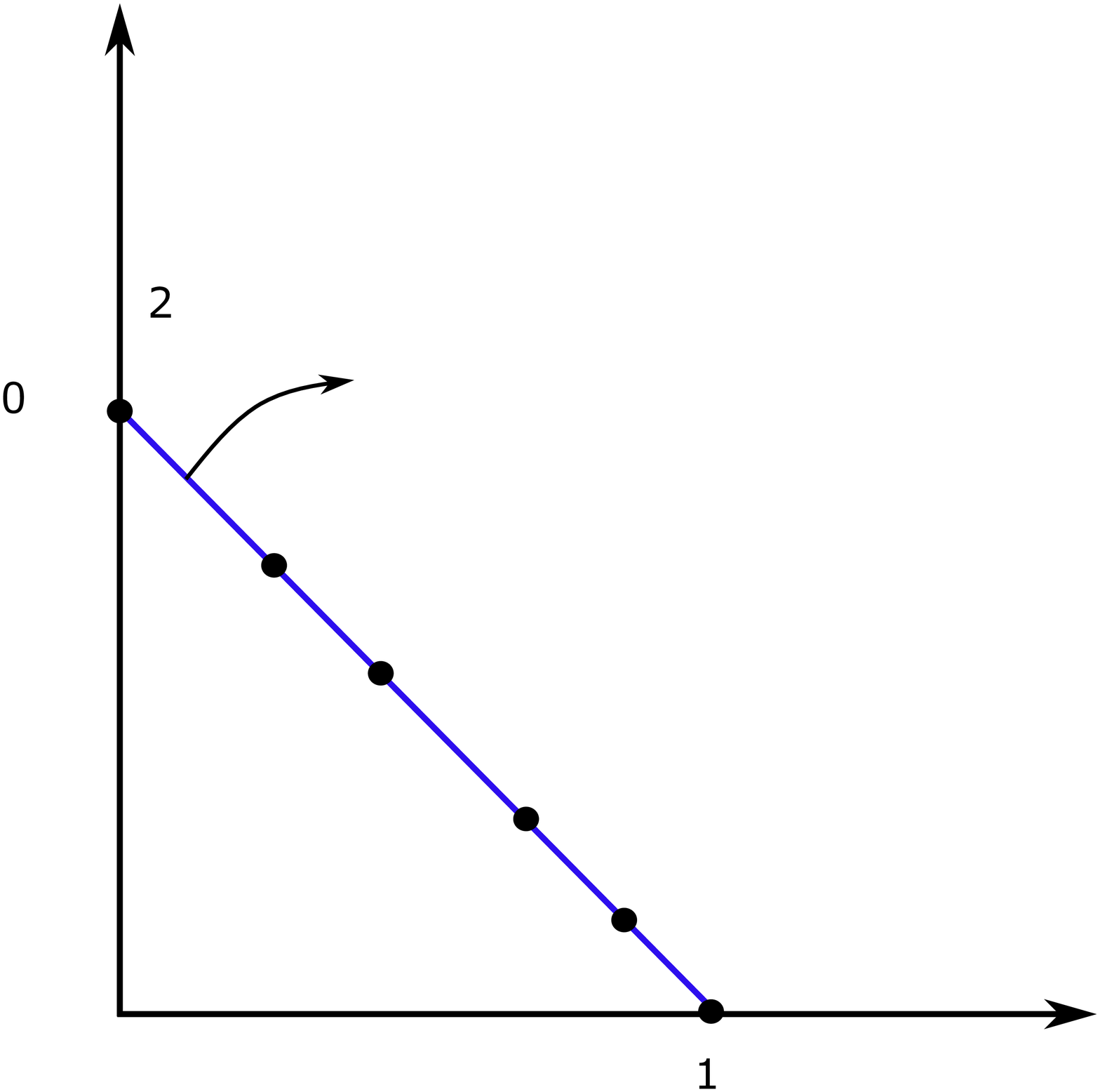}}
  			\vspace{2ex}
  			\centerline{(ii)\;Newton diagram.}
  		\end{minipage}
  		\caption{Configuration of the support of $b\left(\mathcal{X}\right)$ and Newton diagram of $ b\left(\mathcal{X}\right)$.}\label{fig1}
  \end{figure}

  For the vector field $ b\left(\mathcal{X}\right)$, the lowest-degree homogeneous terms of type $\left(1,1\right)$  are
  \begin{align*}
  	&\frac{1}{2}\mathcal{F}_{2d+1}=\frac{1}{2}\left(
  	\begin{array}{l}
  		\left(u^2-v^2\right)\partial_v\left(f_d^2+g_d^2\right)
  		-2uv\partial_u\left(f_d^2+g_d^2\right)\vspace{1ex}\\
  		\left(u^2-v^2\right)\partial_u\left(f_d^2+g_d^2\right)
  		+2uv\partial_v\left(f_d^2+g_d^2\right)\\
  	\end{array}
  	\right)^T.
  \end{align*}
Thus the Hamiltonian associated to the  edge of type $\left(1, 1\right)$ is
 $$h_{2\left(d+1\right)}^{\left(1,1\right)}=\frac{d}{2\left(d+1\right)}\left(u^2+v^2\right)\left(f_d^2\left(u,v\right)+g_d^2\left(u,v\right)\right).$$
  Assume that $h_{2\left(d+1\right)}^{\left(1,1\right)}$ has a factor $v-\lambda u$ with $\lambda\in\mathbb{R}\setminus\{0\}$. Namely, $f_d^2\left(u,v\right)+g_d^2\left(u,v\right)$
  has a factor $v-\lambda u$ with $\lambda\in\mathbb{R}\setminus\{0\}$. Therefore, $v-\lambda u| f_d\left(u,v\right)$  and $v-\lambda u| g_d\left(u,v\right)$. This means that
  $H_{2d,x}$ and $H_{2d,y}$  have the factor  $y-\lambda x$ in common, which contradicts with the hypothesis of the theorem.  Thus, the Hamiltonian  $h_{2\left(d+1\right)}^{\left(1,1\right)}$ has no factor of the form $v-\lambda u$ with $\lambda\in\mathbb{R}\setminus\{0\}$.

 {\bf Case 2: $\deg H_x>\deg H_y$.}  In this case, one has that $\deg H_x=2d-1$ and $\deg H_y=\bar{d}_1-1$ with $2d>\bar{d}_1>1$. Then $H_{2d,x}\not\equiv0$ and $H_{2d,y}\equiv0$.
 It is clear that the Hamiltonian $H\left(x,y\right)$ has the form
 \begin{align}\label{eq12}
 &H\left(x,y\right)=\frac{1}{2}\left(H_{2d}\left(x\right)+H_{2d-1}\left(x\right)+\cdots+H_{\bar{d}_1+1}\left(x\right)+H_{\bar{d}_1}\left(x,y\right)+\cdots+H_2\left(x,y\right)\right),
 \end{align}
 where $H_{j}\left(x\right)=a_jx^j$ for $j=\bar{d}_1+1,\ldots,2d$, with $a_{2d}>0$, $H_{\bar{d}_1,y}\left(x,y\right)\not\equiv0$ and  $H_{i}\left(x,y\right)$  is a homogeneous polynomial of degree $i$ for $i=2,\ldots,\bar{d}_1$. The higher homogeneous terms of the polynomials $ff_x+gg_x$ and $ff_y+gg_y$ are $da_{2d}x^{2d-1}$ and $H_{\bar{d}_1,y}\left(x,y\right)/2$, respectively. Since $x^{2d-1}$ and $H_{\bar{d}_1,y}\left(x,y\right)$ do not have real linear factors in common, we must have
 \begin{align}\label{eq13}
 &H_{\bar{d}_1}\left(x,y\right)=by^{\bar{d}_1}+\sum_{i+j=\bar{d}_1,\;i\geq1}b_{ij}x^iy^j
 \end{align}
 with $b\neq0$. From the equations \eqref{eq8}, \eqref{eq12} and \eqref{eq13}, one has
 \begin{align*}
 \frac{f^2(0,y)+g^2(0,y)}{2}=H(0,y)=\frac{1}{2}\left(by^{\bar{d}_1}+\cdots+H_2(0,y)\right).
 \end{align*}
 Thus, $\bar{d}_1$ is even and $b>0$, and the equation \eqref{eq13} can be rewritten as
  \begin{align}\label{eq15}	&H_{\bar{d}_1}\left(x,y\right)=H_{2d_1}\left(x,y\right)=by^{2d_1}+\sum_{i+j=2d_1,\;i\geq1}b_{ij}x^iy^j
 \end{align}
with $b>0$.

By the equation \eqref{eq-5}, the expression of $b\left(\mathcal{X}\right)$ is given by
\begin{equation}\label{eq18}
	\begin{split}
		&\dot{u}=\frac{1}{2}\sum_{i=2}^{2d}\left(u^2+v^2\right)^{2d-i}\left(\left(u^2-v^2\right)H_{i,v}\left(u,v\right)-2uvH_{i,u}\left(u,v\right)\right),\\
		&\dot{v}=\frac{1}{2}\sum_{i=2}^{2d}\left(u^2+v^2\right)^{2d-i}\left(\left(u^2-v^2\right)H_{i,u}\left(u,v\right)+2uvH_{i,v}\left(u,v\right)\right),
	\end{split}
\end{equation}
where $H_i,\ i=2,3,\ldots,2d$ are given in \eqref{eq12}.
Let
\begin{align}\label{eq19}
	&\mathcal{F}_{4d+1-i}=\left(
	\begin{array}{l}
	\left(u^2+v^2\right)^{2d-i}\left(\left(u^2-v^2\right)H_{i,v}\left(u,v\right)-2uvH_{i,u}\left(u,v\right)\right)\\
		\left(u^2+v^2\right)^{2d-i}\left(\left(u^2-v^2\right)H_{i,u}\left(u,v\right)+2uvH_{i,v}\left(u,v\right)\right)\\
	\end{array}
	\right)^T\in\mathcal{Q}_{4d+1-i}^{\left(1,1\right)}
\end{align}
for $i=2,3, \ldots,2d$.  Then,
$$\text{supp}\left(b\left(\mathcal{X}\right)\right)=\bigcup_{i=2}^{2d}\text{supp}\left(\mathcal{F}_{4d+1-i}\right).$$
By the statement $\rm (a)$ of Lemma \ref{le-2}, $\text{supp}\left(\mathcal{F}_{4d+1-i}\right)$ lie on the straight line $l_{4d+1-i}: x+y=4d+2-i$ for $i=2,3,\ldots,2d$.  From Lemma \ref{le-4}, the Newton diagrams of $\mathcal{F}_{4d+1-i}$ and $\left(u^2+v^2\right)^{2d+1-i}H_i\left(u,v\right)$ have the same vertices for $i=2,3,\ldots,2d$.

Furthermore, we have $\left(u^2+v^2\right)^{2d+1-i}H_i\left(u,v\right)=\left(u^2+v^2\right)^{2d+1-i}H_i\left(u\right)
=a_iu^i\left(u^2+v^2\right)^{2d+1-i}$ for $i=2d_1+1,\ldots,2d$. Consequently, the vertices of the Newton diagram of $\mathcal{F}_{4d+1-i}$ are
$$V_{1,i}=\left(2\left(2d+1\right)-i,0\right)\quad\text{and}\quad V_{2,i}=\left(i,2\left(2d+1-i\right)\right)$$
for $i=2d_1+1,\ldots,2d$.  Note  that $V_{2,i}$ lies on the straight line $l_0:2x+y=2\left(2d+1\right)$  for $i=2d_1+1,\ldots,2d$.

For $i=2d_1$, the Newton diagram of $\mathcal{F}_{4d+1-2d_1}$ has an exterior vertex
 $$V_{2,2d_1}=\left(0,2\left(2d+1-d_1\right)\right)$$
 by the equation \eqref{eq15}. Based on the above analysis, we get the configuration of the support  for the vector field $b\left(\mathcal{X}\right)$, see (i) of Figure \ref{fig2}.
Hence, the Newton diagram of  $ b\left(\mathcal{X}\right)$ has two exterior vertices $V_{1,2d}=\left(2\left(d+1\right),0\right)$ and $V_{2,2d_1}=\left(0,2\left(2d+1-d_1\right)\right)$, and an inner vertex $V_{2,2d}=\left(2d,2\right)$, and two edges of type $\left(2d-d_1,d\right)$ and $\left(1, 1\right)$,  as shown in (ii) of Figure \ref{fig2}.  The vector
 fields associated to the vertices $V_{1,2d}$, $V_{2,2d}$ and $V_{2,2d_1}$  are   $\left(0,a_{2d}du^{2d+1}\right)$ , $\left(-2a_{2d}du^{2d}v,-a_{2d}du^{2d-1}v^2\right)$ and $\left(-bd_1v^{2\left(2d-d_1\right)+1},0\right)$, respectively.
 \begin{figure}[H]
 	{\tiny
 	\hspace{-20pt}\centering
 	\begin{minipage}{0.4\linewidth}
 		\centering
 		\psfrag{0}{$l_0$}
 		\psfrag{1}{$2\left(2d+1\right)$}
 		\psfrag{2}{$2d+1$}
 		\psfrag{3}{$V_{1,2d}$}
 		\psfrag{4}{$V_{2,2d}$}
 		\psfrag{5}{$V_{2,2d_1+1}$}
 		\psfrag{6}{$V_{1,2d_1+1}$}
 		\psfrag{7}{$V_{2,2d_1}$}
 		\psfrag{8}{$l_{2d+1}$}
 		\psfrag{9}{$l_{4d-2d_1}$}
 		\psfrag{10}{$l_{4d-2d_1+1}$}
 		\psfrag{11}{$l_{4d-1}$}
 		\centerline{\includegraphics[width=1\textwidth]{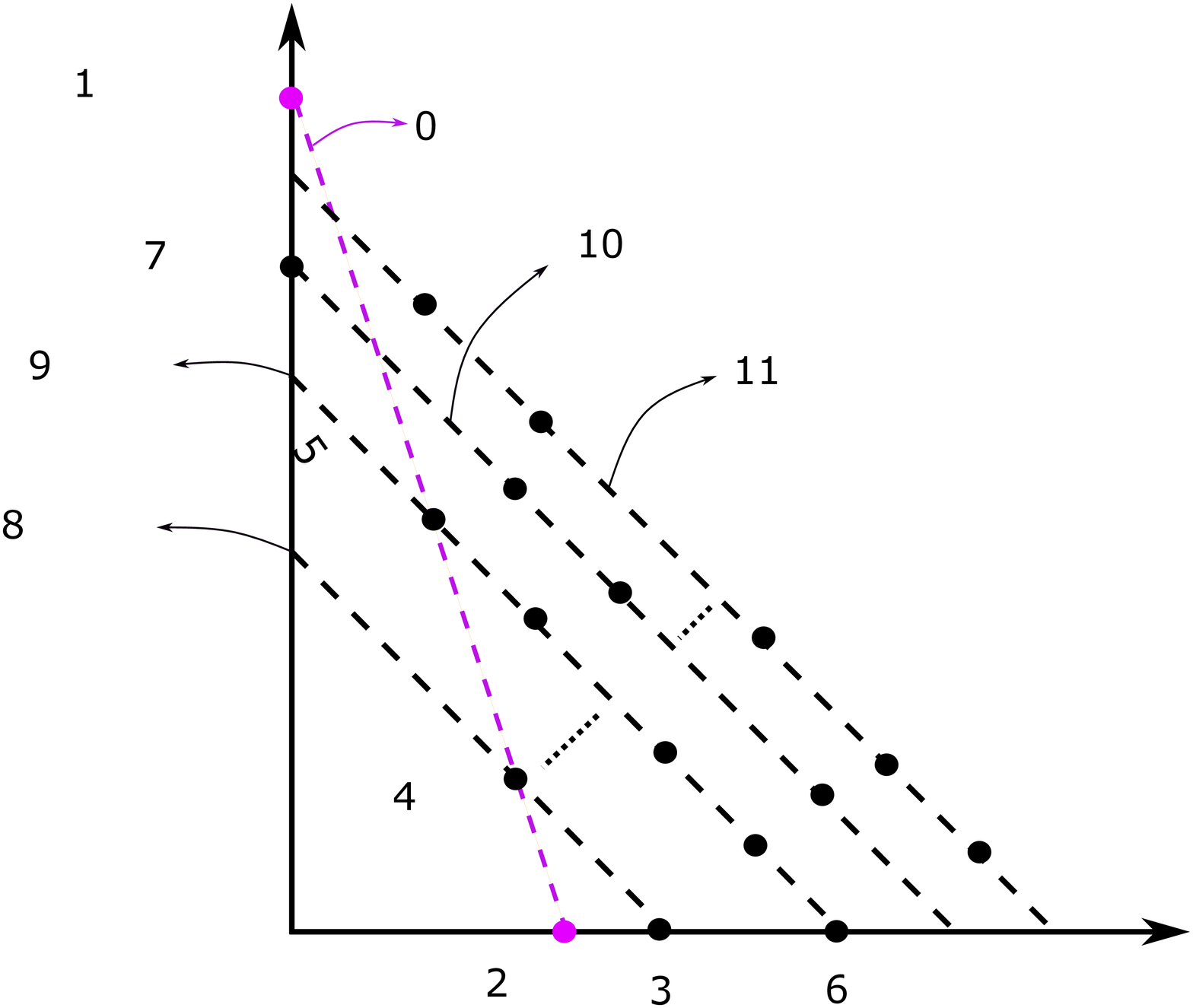}}
 		\vspace{1ex}
 		\centerline{(i)\;Configuration of the support of $ b\left(\mathcal{X}\right)$.}
 			\end{minipage}
 \hspace{-20pt}
 		\begin{minipage}{0.15\linewidth}
 			\centering
 			\centerline{\includegraphics[width=0.4\textwidth]{row.eps}}
 		\end{minipage}
 \hspace{-20pt}
 	\begin{minipage}{0.35\linewidth}
 	\centering
 	\psfrag{3}{$V_{1,2d}$}
 	\psfrag{4}{$V_{2,2d}$}
 	\psfrag{7}{$V_{2,2d_1}$}
 	\psfrag{10}{$h_{2d\left(2d-d_1+1\right)}^{\left(2d-d_1,d\right)}\!=\!\frac{d_1}{2\left(2d-d_1+1\right)}\left(a_{2d}u^{2d}\!+\!bv^{2\left(2d-d_1\right)}\right)v^2$}
 	\psfrag{11}{$h_{2\left(d+1\right)}^{\left(1,1\right)}=\frac{a_{2d}}{2\left(d+1\right)}\left(u^2+v^2\right)u^{2d}$}
 	\centerline{\includegraphics[width=1\textwidth]{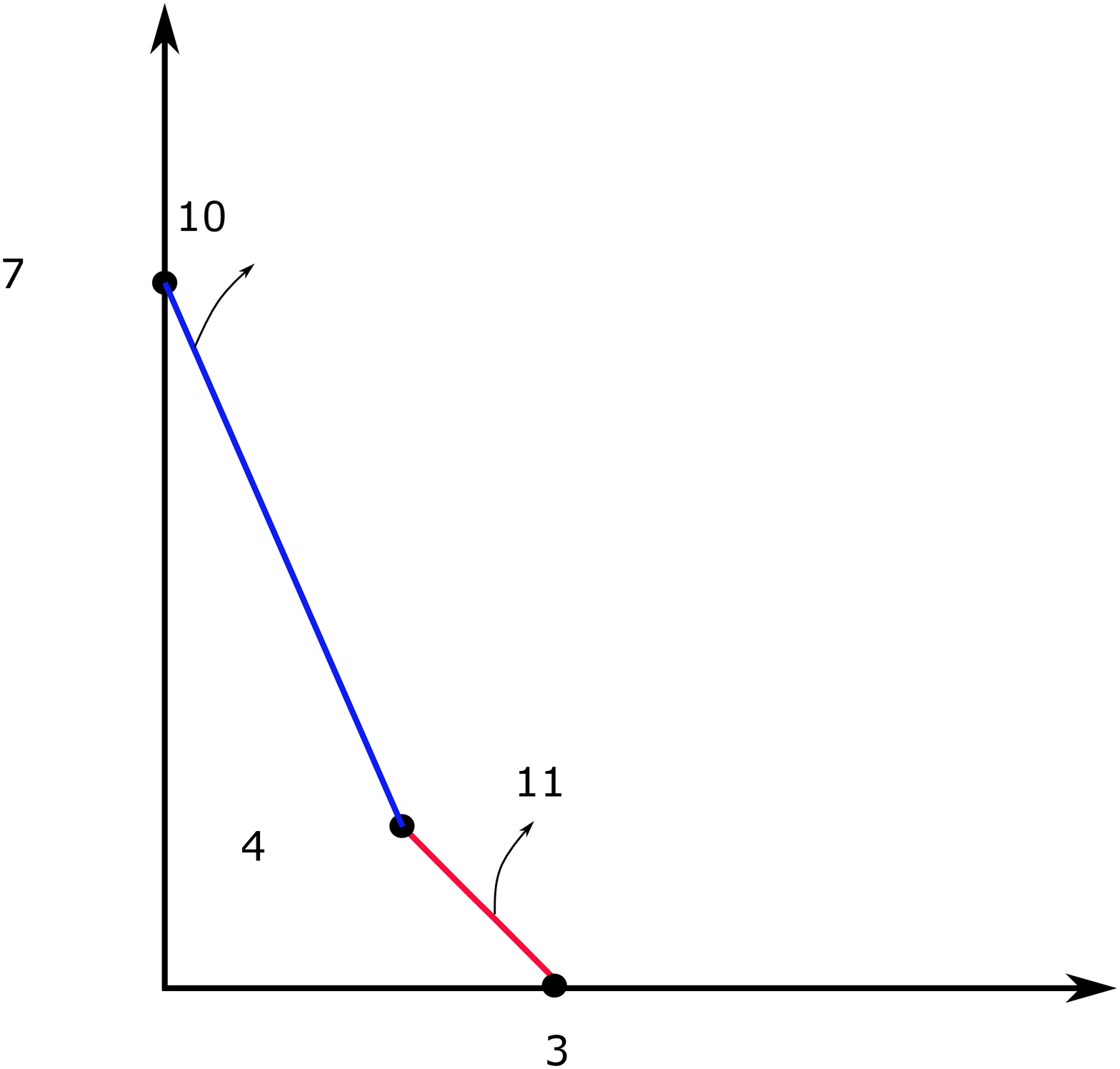}}
 	\vspace{2ex}
 	\centerline{(ii)\;Newton diagram.}
 		\end{minipage}
 	\caption{Configuration of the support of $ b\left(\mathcal{X}\right)$ and Newton diagram for $ b\left(\mathcal{X}\right)$.}\label{fig2}}
 \end{figure}

For the vector field $ b\left(\mathcal{X}\right)$, the lowest-degree homogeneous terms of type $\left(1,1\right)$  are
	\begin{align*}
		&\frac{1}{2}\mathcal{F}_{2d+1}=\frac{1}{2}\left(
		\begin{array}{l}
			\left(u^2-v^2\right)H_{2d,v}\left(u,v\right)
			-2uvH_{2d,u}\left(u,v\right)\vspace{1ex}\\
			\left(u^2-v^2\right)H_{2d,u}\left(u,v\right)
			+2uvH_{2d,v}\left(u,v\right)\\
		\end{array}
		\right)^T=
		\left(
		\begin{array}{l}
			-2a_{2d}du^{2d}v\vspace{1ex}\\
			a_{2d}d\left(u^2-v^2\right)u^{2d-1}
		\end{array}
	\right)^T.
	\end{align*}
Then, the Hamiltonian associated to the edge of type $\left(1, 1\right)$ is
$$h_{2\left(d+1\right)}^{\left(1,1\right)}=\frac{a_{2d}d}{2\left(d+1\right)}\left(u^2+v^2\right)u^{2d}.$$

To obtain the lowest-degree quasi-homogeneous terms of type $\left(2d-d_1,d\right)$, we must determine that whether
there exists some other vertex $V_{2,i}$ which belongs to the edge $\overline{V_{2,2d_1}V_{2,2d}}$.
We again remark that the vertex  $V_{2,i}$  lies on the straight line $l_0:2x+y=2\left(2d+1\right)$ for $i=2d_1+1,\ldots,2d$, see (i) of Figure \ref{fig2}.  This implies that
$$\overline{V_{2,2d_1}V_{2,2d}}\;\cap\;\text{supp}\left(b\left(\mathcal{X}\right)\right)=\left\{V_{2,2d_1},V_{2,2d}\right\}.$$
Thus, the lowest-degree quasi-homogeneous terms of type $\left(2d-d_1,d\right)$ for the vector field
$b\left(\mathcal{X}\right)$ are
	\begin{align*}
	&\left(
	\begin{array}{l}
		-bd_1v^{2\left(2d-d_1\right)+1}-2a_{2d}du^{2d}v\vspace{1ex}\\
		-a_{2d}du^{2d-1}v^2
	\end{array}
	\right)^T.
\end{align*}
The Hamiltonian associated to the  edge of type $\left(2d-d_1,d\right)$ is
$$h_{2d\left(2d-d_1+1\right)}^{\left(2d-d_1,d\right)}=\frac{d_1}{2\left(2d-d_1+1\right)}\left(a_{2d}\left(u^{d}\right)^2+b\left(v^{2d-d_1}\right)^2\right)v^2$$
with $a_{2d}>0$ and $b>0$. It follows that the Hamiltonians  $h_{2d\left(2d-d_1+1\right)}^{\left(2d-d_1,d\right)}$ and $h_{2\left(d+1\right)}^{\left(1,1\right)}$ have no factor of the form $v^{t_1}-\lambda u^{t_2}$ with $\lambda\in\mathbb{R}\setminus\{0\}$.

 {\bf Case 3: $\deg H_y>\deg H_x$.} The proof of this case can be done in the same way as {\bf Case 2} by interchanging the variables $x$ and $y$.

 We complete the proof of Theorem \ref{th2}.
\end{proof}

\section{Proof of Examples \ref{ex1} and \ref{ex2}}\label{se-4}

In this section, we solve Examples \ref{ex1} and \ref{ex2} for the
application of Theorem \ref{th-11}.

\begin{proof}[\bf Proof of Example \ref{ex1}]
Let $\bar{g}\left(x\right)=\sum_{i=1}^mb_ix^{2i}$.  Then, the Hamiltonian of \eqref{eq-1} is given by
\begin{align*}
		&H\left(x,y\right)=\frac{f^2\left(x\right)+\left(y+\bar{g}\left(x\right)\right)^2}{2}=\frac{1}{2}\left(f^2\left(x\right)+\bar{g}^2\left(x\right)+2y\bar{g}\left(x\right)+y^2\right),
	\end{align*}
where $f\left(x\right)=\sum_{i=0}^na_ix^{2i+1}$. Since $2\left(2n+1\right)=\max\left\{\deg f^2, \deg\bar{g}^2\right\}>\deg \left(y\bar{g}\right)=2m+1>2$,  the Hamiltonian $H\left(x,y\right)$ can be written in the form
\begin{equation}\begin{split}\label{eq21}
H\left(x,y\right)=&\frac{1}{2}\sum_{i=2}^{2\left(2n+1\right)}H_i(x,y)\\ =&\frac{1}{2}\left(H_{2\left(2n+1\right)}\left(x\right)\!+\!H_{4n}\left(x\right)\!+\!\cdots\!+\!H_{2\left(m+2\right)}
\left(x\right)\!+\!H_{2(m+1)}\left(x\right)\!+\!H_{2m+1}\left(x,y\right)\!+\!
\!\cdots\!+\!H_2\left(x,y\right)\right),
\end{split}\end{equation}	
where $H_i(x,y)$ is a homogeneous polynomial of degree $i$, and $H_i(x,y)=H_i(x)=c_ix^i$ for $i=2\left(m+1\right),2\left(m+2\right),\ldots,4n,2\left(2n+1\right)$, with $c_{2\left(2n+1\right)}=a_n^2>0$,  and for  $i=2,3,\ldots,2m+1$,
\begin{equation}\label{eq23}
H_i\left(x,y\right)=
\begin{cases}
(a_0^2+b_1^2)x^2+y^2,\;\text{if}\;i=2,\\
c_ix^i,\;\text{if}\;i\;\text{is even with}\;i\geq4,\\
2b_{(i-1)/2}x^{i-1}y,\;\text{if}\;i\;\text{is odd}.
\end{cases}
\end{equation}

Using the equation \eqref{eq-5}, the compactified vector field $b\left(\mathcal{X}\right)$ is
\begin{equation}\label{eq22}
	\begin{split}
		&\dot{u}=\frac{1}{2}\sum_{i=2}^{2\left(2n+1\right)}\left(u^2+v^2\right)^{2\left(2n+1\right)-i}\left(\left(u^2-v^2\right)H_{i,v}\left(u,v\right)-2uvH_{i,u}\left(u,v\right)\right),\\
		&\dot{v}=\frac{1}{2}\sum_{i=2}^{2\left(2n+1\right)}\left(u^2+v^2\right)^{2\left(2n+1\right)-i}\left(\left(u^2-v^2\right)H_{i,u}\left(u,v\right)+2uvH_{i,v}\left(u,v\right)\right),
	\end{split}
\end{equation}
where $H_i,\ i=2,3,\ldots,2m+1,2\left(m+1\right),2\left(m+2\right),\ldots,2\left(2n+1\right)$ are given in \eqref{eq21}.  By a similar analysis as {\bf Case 2} in the proof of Theorem \ref{th2}, we obtain the following statements:
\begin{itemize}
\item [(i)] The support of $\mathcal{F}_{8n+5-i}$ (see \eqref{eq19}) lies on the straight line $l_{8n+5-i}: x+y=2\left(4n+3\right)-i$ for $i=2,3,\ldots,2m+1,2\left(m+1\right),2\left(m+2\right),\ldots,2\left(2n+1\right)$.
\item [(ii)] If $i=2$, then the Newton diagram of $\mathcal{F}_{8n+3}$ has an exterior vertex $V_{2,2}=\left(0,4\left(2n+1\right)\right)$.
\item [(iii)] If $i$ is even with $i\geq4$ and $c_i\neq0$, then the vertices of the Newton diagram of $\mathcal{F}_{8n+5-i}$ are
$$V_{1,i}=\left(2\left(4n+3\right)-i,0\right)\quad\text{and}\quad V_{2,i}=\left(i,2\left(4n+3-i\right)\right).$$
Moreover, $V_{2,i}$ lies on the straight line $\ell_0:2x+y=2\left(4n+3\right)$.
\item [(iv)] If $i$ is odd with $3\leq i\leq2m+1$ and $b_{(i-1)/2}\neq0$, then the vertices of the Newton diagram of $\mathcal{F}_{8n+5-i}$ are
\begin{align}\label{eq26}
V_{1,i}=\left(8n+5-i,1\right)\quad\text{and}\quad V_{2,i}=\left(i-1,8n+7-2i\right).
\end{align}
Moreover, $V_{2,i}$ lies on the straight line $\ell_1:2x+y=8n+5$.
\end{itemize}

  \begin{figure}[H]
  	 	{\tiny
	\centering
	\begin{minipage}{0.42\linewidth}
		\centering
		\psfrag{0}{$V_{2,2(2n+1)}$}
		\psfrag{1}{$V_{1,2(2n+1)}$}
		\psfrag{2}{$l_{4n+3}$}
		\psfrag{3}{$l_{2\left(4n-m\right)+3}$}
		\psfrag{4}{$l_{2\left(4n-m+2\right)}$}
		\psfrag{5}{$V_{2,2}$}
		\psfrag{6}{$8n+5$}
		\psfrag{7}{$2\left(4n+3\right)$}
		\psfrag{8}{$l_{8n+3}$}
		\psfrag{9}{$\ell_1$}
		\psfrag{10}{$\ell_0$}
		\psfrag{11}{$V_{2,2m+1}$}
		\centerline{\includegraphics[width=1\textwidth]{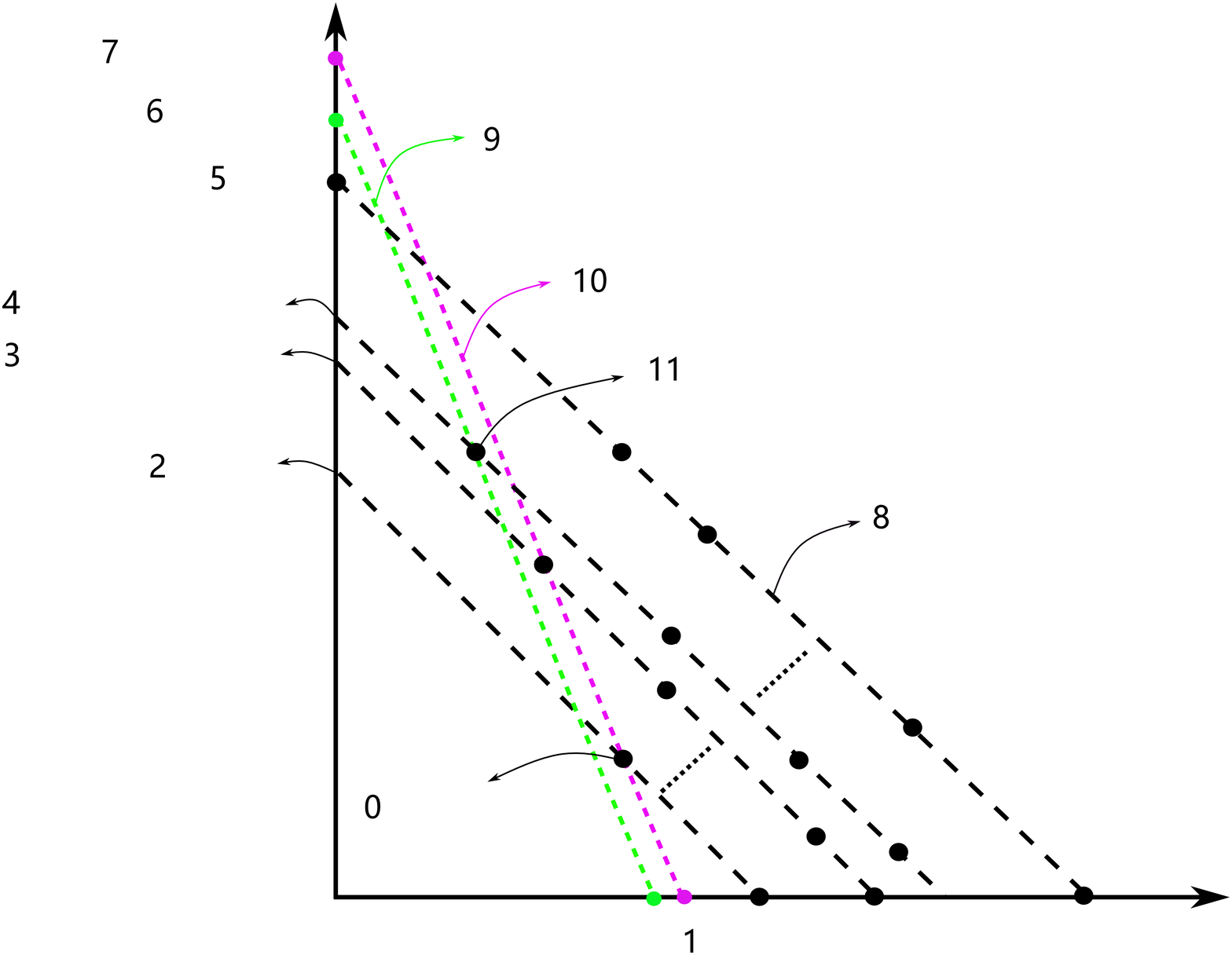}}
		\vspace{1ex}
		\centerline{(i)\;Configuration of the support of system \eqref{eq22}.}
	\end{minipage}
	\begin{minipage}{0.15\linewidth}
		\centering
		\centerline{\includegraphics[width=0.4\textwidth]{row.eps}}
	\end{minipage}
	\begin{minipage}{0.27\linewidth}
		\centering
			\psfrag{0}{$V_{2,2(2n+1)}$}
		\psfrag{1}{$V_{1,2(2n+1)}$}
		\psfrag{2}{$h_{4(2n+1)^2}^{\left(4n+1,2n+1\right)}$}
		\psfrag{3}{$h_{4\left(n+1\right)}^{\left(1,1\right)}$}
		\psfrag{5}{$V_{2,2}$}
		\centerline{\includegraphics[width=1\textwidth]{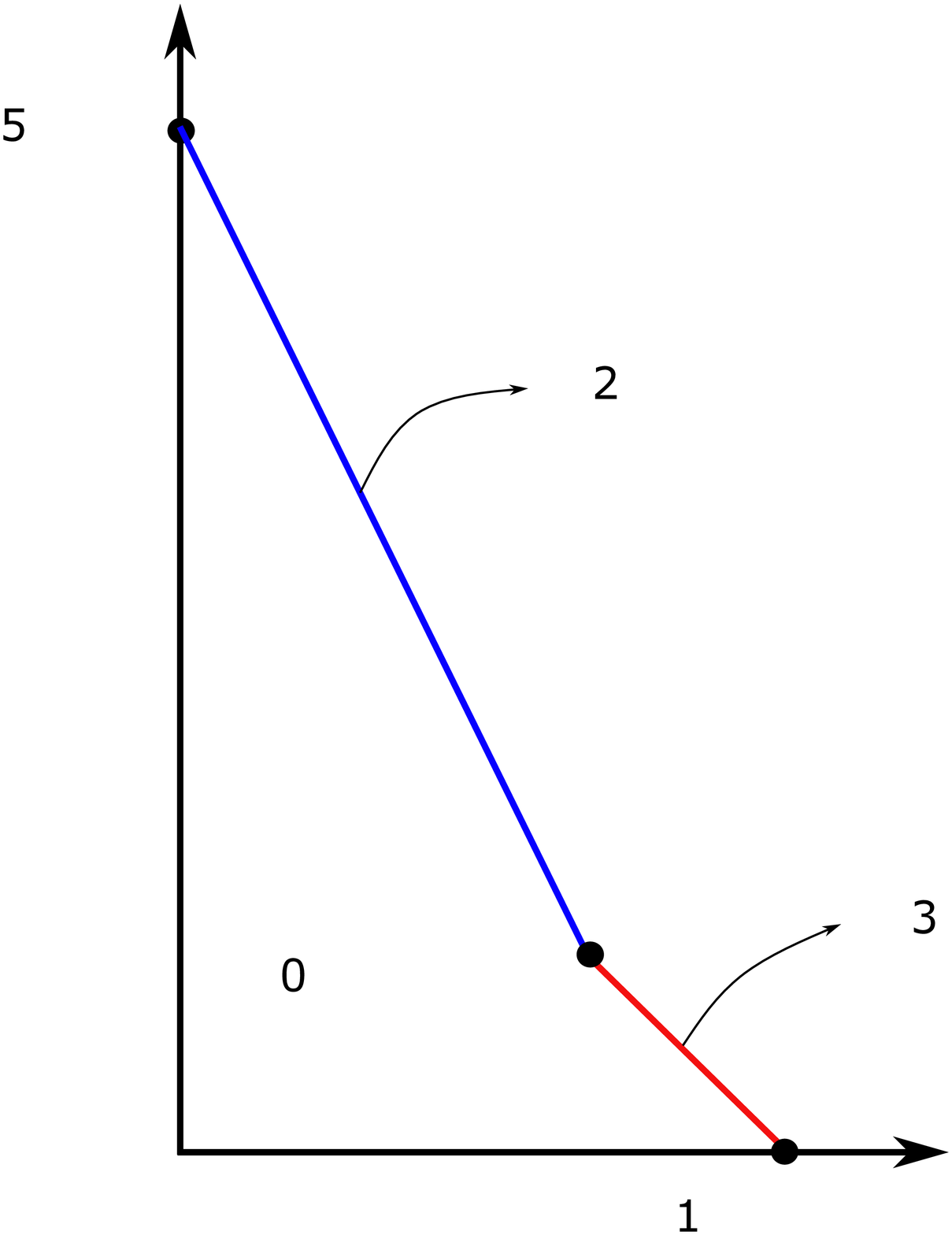}}
		\vspace{2ex}
		\centerline{(ii)\;Newton diagram.}
	\end{minipage}
	\caption{Configuration of the support of system \eqref{eq22} and Newton diagram of system \eqref{eq22}.}\label{fig3}}
\end{figure}

Therefore,  the configuration of the support for the vector field \eqref{eq22} is given in (i) of Figure \ref{fig3}. By the statement (d) of Theorem \ref{th-9}, the vertex $V_{2,i}$  is not a vertex of the Newton diagram of system \eqref{eq22} if $i$ is odd with $3\leq i\leq2m+1$ and $b_{(i-1)/2}\neq0$. We get that
the Newton diagram of system \eqref{eq22} has two exterior vertices $V_{2,2}=\left(0,4\left(2n+1\right)\right)$ and $V_{1,2(2n+1)}=\left(4\left(n+1\right),0\right)$, and an inner vertex $V_{2,2(2n+1)}=\left(2\left(2n+1\right),2\right)$, and two edges of type $\left(4n+1,2n+1\right)$ and $\left(1, 1\right)$,  see (ii) of  Figure \ref{fig3}.   The vector
fields associated to the vertices $V_{2,2}$, $V_{2,2(2n+1)}$ and $V_{1,2(2n+1)}$  are   $\left(-v^{8n+3},0\right)$, $\left(-2\left(2n+1\right)c_{2\left(2n+1\right)}u^{2\left(2n+1\right)}v,-\left(2n+1\right)c_{2\left(2n+1\right)}u^{4n+1}v^2\right)$ and $\left(0,\left(2n+1\right)c_{2\left(2n+1\right)}u^{4n+3}\right)$, respectively.

For the vector field \eqref{eq22}, the lowest-degree homogeneous terms of type $\left(1,1\right)$  are
\begin{align*}
	&\frac{1}{2}\mathcal{F}_{4n+3}=
	\left(
	\begin{array}{l}
		-2\left(2n+1\right)c_{2\left(2n+1\right)}u^{2\left(2n+1\right)}v\vspace{1ex}\\
	\left(2n+1\right)c_{2\left(2n+1\right)}\left(u^2-v^2\right)u^{4n+1}
	\end{array}
	\right)^T.
\end{align*}
Thus the Hamiltonian associated to the edge of type $\left(1, 1\right)$ is
$$h_{4\left(n+1\right)}^{\left(1,1\right)}=\frac{\left(2n+1\right)c_{2\left(2n+1\right)}}{4\left(n+1\right)}\left(u^2+v^2\right)u^{2\left(2n+1\right)}.$$

Next, we compute the Hamiltonian associated to the edge of type $\left(4n+1,2n+1\right)$. The line $\ell_2$ passing through $V_{2,2}$ and $V_{2,2(2n+1)}$ is $\left(4n+1\right)x+\left(2n+1\right)y=4\left(2n+1\right)^2$.  The lines $\ell_1$ and $\ell_2$ intersect at  the point  $\left(2n+1,4n+3\right)$.  If $\left(2n+1,4n+3\right)\in\text{supp}\left(b\left(\mathcal{X}\right)\right)$, by \eqref{eq26}, then
there exists an odd number $i$ such that $i-1=2n+1$, which is impossible. So,
$$\ell_2\cap\text{supp}\left(b\left(\mathcal{X}\right)\right)=\left\{V_{2,2},V_{2,2(2n+1)}\right\}.$$
The lowest-degree quasi-homogeneous terms of type $\left(4n+1,2n+1\right)$ for the vector field
\eqref{eq22} are
\begin{align*}
	&\left(
	\begin{array}{l}
		-v^{8n+3}-2\left(2n+1\right)c_{2\left(2n+1\right)}u^{2\left(2n+1\right)}v\vspace{1ex}\\
		-\left(2n+1\right)c_{2\left(2n+1\right)}u^{4n+1}v^2
	\end{array}
	\right)^T.
\end{align*}
It follows that the Hamiltonian associated to the  edge of type  $\left(4n+1,2n+1\right)$ is
$$h_{4(2n+1)^2}^{\left(4n+1,2n+1\right)}=\frac{1}{4(2n+1)}\left(c_{2\left(2n+1\right)}\left(u^{2n+1}\right)^{2}+\left(v^{4n+1}\right)^2\right)v^2$$
with $c_{2\left(2n+1\right)}>0$. It is obvious that the Hamiltonians  $h_{4(2n+1)^2}^{\left(4n+1,2n+1\right)}$ and $h_{4\left(n+1\right)}^{\left(1,1\right)}$ have no factor of the form $v^{t_1}-\lambda u^{t_2}$ with $\lambda\in\mathbb{R}\setminus\{0\}$. Using Theorem \ref{th-11}, the polynomial map $F$ is injective.
\end{proof}

\begin{proof}[\bf Proof of Example \ref{ex2}]
Let $\bar{f}\left(y\right)=\sum_{i=0}^{m_1}a_iy^{2i+1},\tilde{f}\left(x\right)=\sum_{i=0}^{m_2}b_ix^{2i+1},\bar{g}\left(y\right)=\sum_{i=0}^{m_3}c_iy^{2i+1}$ and $\tilde{g}\left(x\right)=\sum_{i=0}^{m_2}d_ix^{2i+1}$.
Then, the Hamiltonian of the vector field \eqref{eq-1} is	
\begin{align*} &H\left(x,y\right)=\frac{f^2+g^2}{2}=\frac{1}{2}\left(\bar{f}^2\left(y\right)+\bar{g}^2\left(y\right)+2\left(\bar{f}\left(y\right)\tilde{f}\left(x\right)-\bar{g}\left(y\right)\tilde{g}\left(x\right)\right)+\tilde{f}^2\left(x\right)+\tilde{g}^2\left(x\right)\right).
\end{align*}
Since $m_1>\max\left\{m_2,m_3\right\}$, we have that
$2\left(2m_1+1\right)=\max\left\{\deg \bar{f}^2,\deg\bar{g}^2\right\}>\max\big\{\deg\left(\bar{f}\tilde{f}\right),$ $\deg\left(\bar{g}\tilde{g}\right)\big\}=2\left(m_1+m_2+1\right)>2\left(2m_2+1\right)=\max\left\{\deg \tilde{f}^2,\deg\tilde{g}^2\right\}$.
 So, the Hamiltonian $H\left(x,y\right)$ can be represented in the form
\begin{small}
\begin{equation}\label{eq25}
\begin{split}
H\left(x,y\right)=&\frac{1}{2}\sum_{k=1}^{2m_1+1}H_{2k}\left(x,y\right)\\[1ex]
=&\frac{1}{2}\Big(a_{m_1}^2y^{2\left(2m_1+1\right)}\!+\!	H_{4m_1}\left(y\right)\!+\!\cdots\!+\!H_{2\left(m_1+m_2+2\right)}\left(y\right)\!+\!\alpha_{2\left(m_1+m_2+1\right)}y^{2\left(m_1+m_2+1\right)}\\[1ex]
&\!+\!2a_{m_1}b_{m_2}x^{2m_2+1}y^{2m_1+1}\!+\!H_{2\left(m_1+m_2\right)}\left(x,y\right)\!+\!\cdots\!+\!H_{4\left(m_2+1\right)}\left(x,y\right)
\!+\!\alpha_{2\left(2m_2+1\right)}y^{2\left(2m_2+1\right)}\\[1ex]
&\!+\!\sum\limits_{i+j=2m_2}\alpha_{i,j}x^{2i+1}y^{2j+1}\!+\!\left(b_{m_2}^2\!+\!d_{m_2}^2\right)x^{2\left(2m_2+1\right)}
\!+\!H_{4m_2}\left(x,y\right)\!+\!\cdots\!+\!H_2\left(x,y\right)\Big),
\end{split}
\end{equation}
\end{small}where $H_{2k}\left(x,y\right)$ is a homogeneous polynomial of degree $2k$ for $k=1,2,\ldots,2m_1+1$ and  $H_{2k}\left(x,y\right)$ is one of the following three types:\\
{\bf Type I}:\ $H_{2k}\left(x,y\right)=\alpha_{2k}y^{2k}$,\\[1ex]
{\bf Type II}:\ $H_{2k}\left(x,y\right)=\alpha_{2k}y^{2k}+\sum\limits_{i+j+1=k}\alpha_{i,j}x^{2i+1}y^{2j+1}$ with $2\left(m_2+1\right)\leq k\leq m_1+m_2+1$,\\[1ex]
{\bf Type III}:\ $H_{2k}\left(x,y\right)=\alpha_{2k}y^{2k}+\sum\limits_{i+j+1=k}\alpha_{i,j}x^{2i+1}y^{2j+1}+\beta_{2k}x^{2k}$ with $k\leq2m_2+1$.\\
Here $0\leq i\leq m_2$ and $0\leq j\leq m_1$.

Applying  the equation \eqref{eq-5}, the Bendixson compactification of $\mathcal{X}$ is given by
\begin{equation}\label{eq27}
	\begin{split}
		&\dot{u}=\frac{1}{2}\sum_{k=1}^{2m_1+1}\left(u^2+v^2\right)^{2\left(2m_1-k+1\right)}\left(\left(u^2-v^2\right)H_{2k,v}\left(u,v\right)-2uvH_{2k,u}\left(u,v\right)\right),\\
		&\dot{v}=\frac{1}{2}\sum_{k=1}^{2m_1+1}\left(u^2+v^2\right)^{2\left(2m_1-k+1\right)}\left(\left(u^2-v^2\right)H_{2k,u}\left(u,v\right)+2uvH_{2k,v}\left(u,v\right)\right),
	\end{split}
\end{equation}
where $H_{2k},\ k=1,2,\ldots,2m_1+1$ are given in \eqref{eq25}. Repeating the process of the proof of {\bf Case 2}  in Theorem \ref{th2}, we have the following results:
\begin{itemize}
	\item [(i)] The support of $\mathcal{F}_{2\left(4m_1-k\right)+5}$ (see \eqref{eq19}) lies on the straight line $l_{2\left(4m_1-k\right)+5}: x+y=2\left(4m_1-k+3\right)$ for $k=1,2,\ldots,2m_1+1$.
	\item [(ii)] If $H_{2k}\left(x,y\right)=\alpha_{2k}y^{2k}$ (i.e., {\bf Type I}) and $\alpha_{2k}\neq0$, then the vertices of the Newton diagram of $\mathcal{F}_{2\left(4m_1-k\right)+5}$ are
	$$V_{1,k}=\left(2\left(4m_1-2k+3\right),2k\right)\quad\text{and}\quad V_{2,k}=\left(0,2\left(4m_1-k+3\right)\right).$$
	Moreover, $V_{1,k}$ lies on the straight line $\ell_0:x+2y=2\left(4m_1+3\right)$.
		\item [(iii)] If $k=2m_2+1$, then $\beta_{2(2m_2+1)}=b_{m_2}^2\!+\!d_{m_2}^2>0$ and the Newton diagram of $\mathcal{F}_{4\left(2m_1-m_2\right)+3}$ has an exterior vertex $V_{1,2m_2+1}=\left(4\left(2m_1-m_2+1\right),0\right)$.
	\item [(iv)] If $H_{2k}\left(x,y\right)=\alpha_{2k}y^{2k}+\sum\limits_{i+j+1=k}\alpha_{i,j}x^{2i+1}y^{2j+1}$ with $0\leq i\leq m_2$ and $0\leq j\leq m_1$ (i.e., {\bf Type II}), then one of the vertices of the Newton diagram of $\mathcal{F}_{2\left(4m_1-k\right)+5}$ is
	$$V_{1,k}=\left(2\left(4m_1-j_k-k\right)+5,2j_k+1\right),$$
where
\begin{equation}\label{iv}
j_k=\min\left\{j\mid i+j+1=k,\ 0\leq i\leq m_2,\ 0\leq j\leq m_1,\ \alpha_{k-1-j,j}\neq0\right\},
\end{equation}
and the other vertex $V_{2,k}=\left(0,2\left(4m_1-k+3\right)\right)$ if $\alpha_{2k}\neq0$, otherwise
$$V_{2,k}=\left(2i_k+1,2\left(4m_1-i_k-k\right)+5\right)$$
where $i_k=\min\left\{i\mid i+j+1=k,\ 0\leq i\leq m_2,\ 0\leq j\leq m_1,\ \alpha_{i,k-1-i}\neq0\right\}.$
\end{itemize}

We remark  that  if $H_{2k}\left(x,y\right)$ has the form {\bf Type II}, then the vertex $V_{1,k}$ in (iv) is not the vertex of the Newton diagram of system \eqref{eq27} by the statement (d) of Theorem \ref{th-9}.  Consequently, the configuration of the support for the vector field \eqref{eq27} is described in (i) of Figure \ref{fig-2}. We  obtain that
the Newton diagram of system \eqref{eq27} has two exterior vertices $V_{2,2m_1+1}=\left(0,4\left(m_1+1\right)\right)$ and $V_{1,2m_2+1}=\left(4\left(2m_1-m_2+1\right),0\right)$, an inner vertex $V_{1,2m_1+1}=\left(2,2\left(2m_1+1\right)\right)$, and two edges of type $\left(1,1\right)$ and $\left(2m_1+1, 2\left(2m_1-m_2\right)+1\right)$, see (ii) of Figure \ref{fig-2}.  Moreover, the vector
fields associated to the vertices $V_{2,2m_1+1}$, $V_{1,2m_1+1}$ and $V_{1,2m_2+1}$  are   $\left(-\left(2m_1+1\right)a_{m_1}^2v^{4m_1+3},0\right)$, $ \left(\left(2m_1+1\right)a_{m_1}^2u^2v^{4m_1+1},2\left(2m_1+1\right)a_{m_1}^2uv^{2\left(2m_1+1\right)}\right)$ and $\left(0,\left(2m_2\!+\!1\right)\left(b_{m_2}^2\!+\!d_{m_2}^2\right)u^{4\left(2m_1-m_2\right)+3}\right)$, respectively.

 \begin{figure}[H]
	{\tiny
		\centering
		\begin{minipage}{0.42\linewidth}
			\centering
			\psfrag{0}{$4m_1+3$}
			\psfrag{1}{$V_{1,2m_1+1}$}
			\psfrag{2}{$V_{1,2m_2+1}$}
			\psfrag{3}{$V_{1,m_1+m_2+2}$}
			\psfrag{4}{$V_{2,2m_1+1}$}
			\psfrag{5}{$l_{2\left(3m_1-m_2\right)+1}$}
			\psfrag{6}{$l_{4\left(2m_1-m_2\right)+3}$}
			\psfrag{7}{$l_{8m_1+3}$}
			\psfrag{8}{$\ell_0$}
			\psfrag{9}{$2\left(4m_1+3\right)$}
			\psfrag{10}{$l$}
			\centerline{\includegraphics[width=1\textwidth]{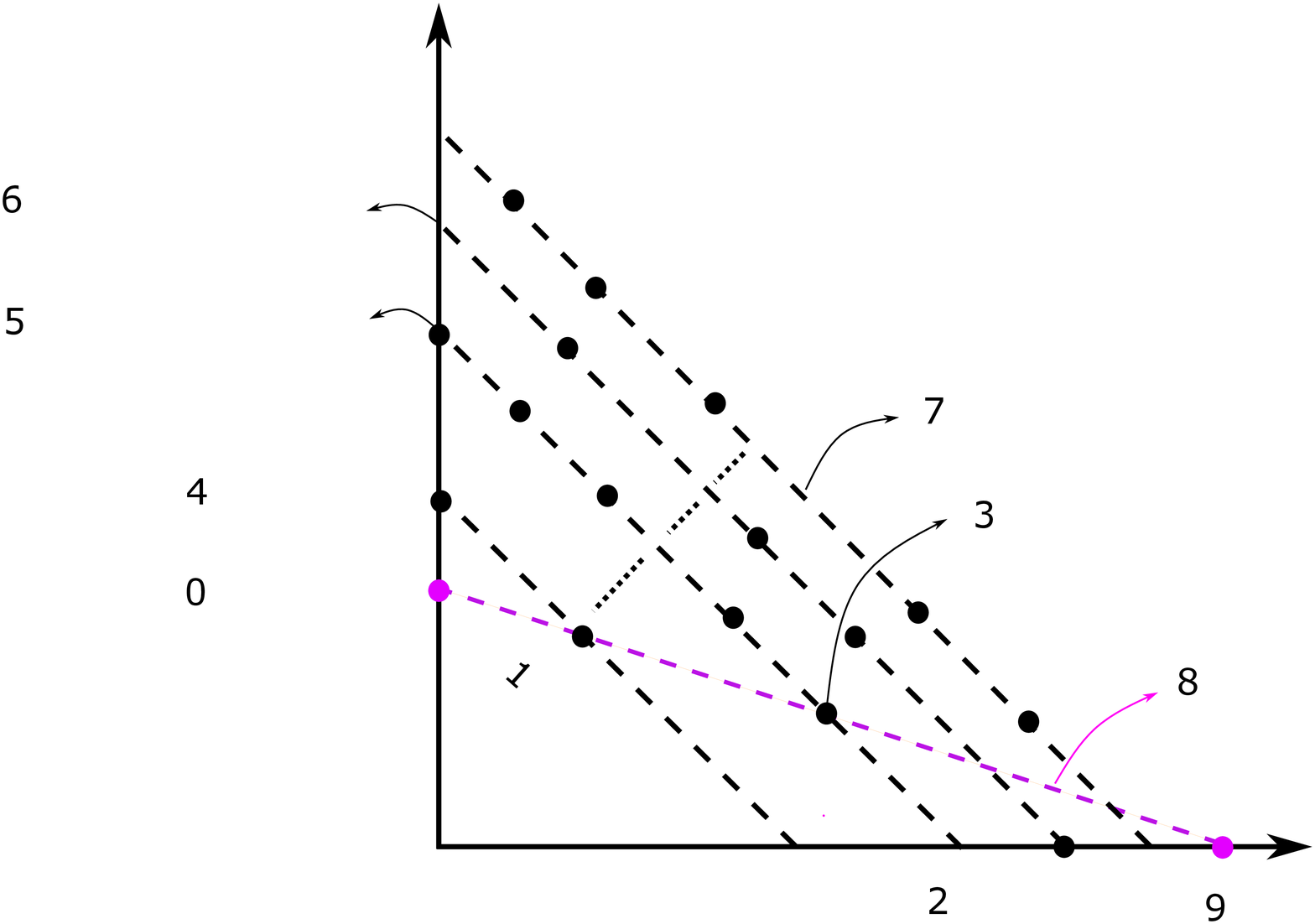}}
			\vspace{1ex}
			\centerline{(i)\;Configuration of the support of system \eqref{eq27}.}
		\end{minipage}
		\begin{minipage}{0.15\linewidth}
			\centering
			\centerline{\includegraphics[width=0.4\textwidth]{row.eps}}
		\end{minipage}
		\begin{minipage}{0.35\linewidth}
			\centering
				\psfrag{1}{$V_{1,2m_1+1}$}
			\psfrag{2}{$V_{1,2m_2+1}$}
			\psfrag{3}{$h_{4\left(m_1+1\right)}^{\left(1,1\right)}$}
			\psfrag{4}{$V_{2,2m_1+1}$}
			\psfrag{5}{$h_{4\left(2m_1+1\right) \left(2m_1-m_2+1\right)}^{\left(2m_1+1, 2\left(2m_1-m_2\right)+1\right)}$}
			\psfrag{6}{$V_{1,m_1+m_2+1}$}
			\centerline{\includegraphics[width=1\textwidth]{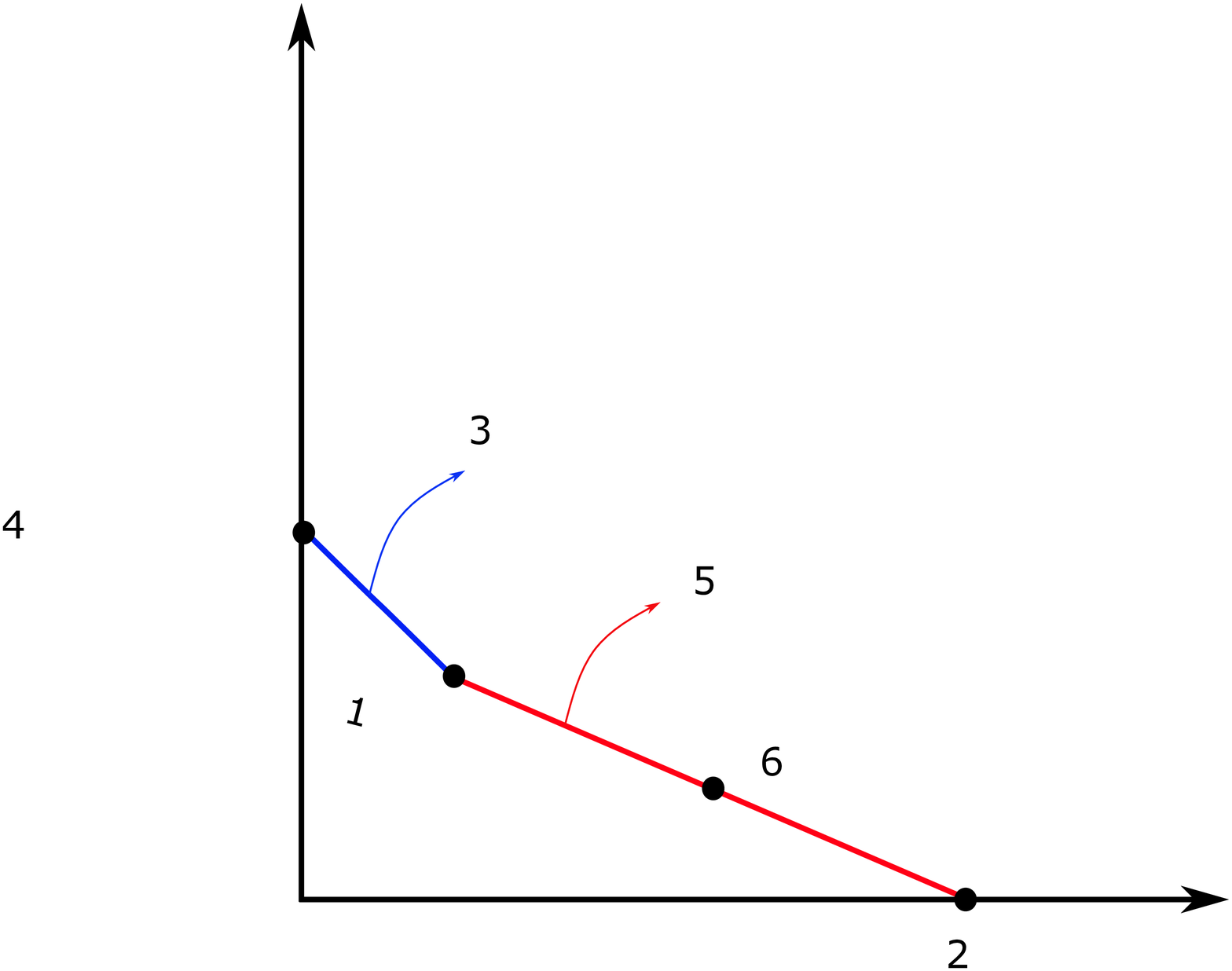}}
			\vspace{2ex}
			\centerline{(ii)\;Newton diagram.}
		\end{minipage}
		\caption{Configuration of the support of system \eqref{eq27} and Newton diagram for system \eqref{eq27}.}\label{fig-2}}
\end{figure}

The lowest-degree homogeneous terms of type $\left(1,1\right)$  are given by
\begin{align*}
	&\frac{1}{2}\mathcal{F}_{4m_1+3}=
	\left(
	\begin{array}{l}
	\left(2m_1+1\right)a_{m_1}^2\left(u^2-v^2\right)v^{4m_1+1}\vspace{1ex}\\
	2\left(2m_1+1\right)a_{m_1}^2uv^{2\left(2m_1+1\right)}
	\end{array}
	\right)^T.
\end{align*}
Thus the Hamiltonian associated to the edge of type $\left(1, 1\right)$ is
$$h_{4\left(m_1+1\right)}^{\left(1,1\right)}=\frac{2m_1+1}{4\left(m_1+1\right)}a_{m_1}^2\left(u^2+v^2\right)v^{2\left(2m_1+1\right)}.$$

Next, we calculate the Hamiltonian associated to the edge of type $\left(2m_1+1, 2\left(2m_1-m_2\right)+1\right)$. Notice that the edge of type $\left(2m_1+1, 2\left(2m_1-m_2\right)+1\right)$ lies on the straight line $\ell_1:\left(2m_1+1\right)x+\left(2\left(2m_1-m_2\right)+1\right)y=4\left(2m_1-m_2+1\right)\left(2m_1+1\right)$.  Let $V_{1,k}$ and $V_{2, k}$ be the vertices of the Newton diagram of $\mathcal{F}_{2\left(4m_1-k\right)+5}$, one has
$$\ell_1\cap\text{supp}\left(b\left(\mathcal{X}\right)\right)\subset\bigcup_{k=1}^{2m_1+1}\left\{V_{1,k},V_{2, k}\right\}.$$
Based on the above discussions and (i) of Figure \ref{fig-2}, we know that if $H_{2k}\left(x,y\right)$ has the expression of {\bf Type I} or {\bf Type III}, then the vertices of the Newton diagram of $\mathcal{F}_{2\left(4m_1-k\right)+5}$ do not belong to the  set
$$\Upsilon=\ell_1\cap\text{supp}\left(b\left(\mathcal{X}\right)\right)\setminus\left\{V_{1,2m_1+1},V_{1,2m_2+1}\right\}.$$
This tells us that if the vertices of the Newton diagram of $\mathcal{F}_{2\left(4m_1-k\right)+5}$ belong to $\Upsilon$, then $H_{2k}\left(x,y\right)$ must be the form of {\bf Type II} and $V_{1,k}=\left(2\left(4m_1-j_k-k\right)+5,2j_k+1\right)\in\Upsilon$. Thereby,
\begin{equation}\label{jk}
\left(2m_1\!+\!1\right)\left(2\left(4m_1\!-\!j_k\!-\!k\right)\!+\!5\right)\!+\!\left(2\left(2m_1\!-\!m_2\right)\!+\!1\right)\left(2j_k\!+\!1\right)\!
=\!4\left(2m_1\!-\!m_2\!+\!1\right)\left(2m_1\!+\!1\right).
\end{equation}
To solve all $j_k$ and $k$ satisfying the equation \eqref{jk}, we let the bivariate function	
$$G(j_k,k)=\left(2m_1\!+\!1\right)\left(2\left(4m_1\!-\!j_k\!-\!k\right)\!+\!5\right)\!+\!\left(2\left(2m_1\!-\!m_2\right)\!+\!1\right)\left(2j_k\!+\!1\right)\!
-\!4\left(2m_1\!-\!m_2\!+\!1\right)\left(2m_1\!+\!1\right),$$
which is defined in the region $\Omega=\{(j_k,k)\mid0\leq j_k\leq m_1$, $j_k+1\leq k\leq j_k+m_2+1\}$, see \eqref{iv}. It is easy to verify that
\begin{equation*}\begin{split}
&G_{j_k}(j_k,k)=4(m_1-m_2)>0,\quad G_{k}(j_k,k)=-2(2m_1+1)<0,\\
&G(0,k)=2(-(2m_1+1)k+3m_1+m_2+4m_1m_2+1)\geq G(0,m_2+1)=2m_1(2m_2+1)>0,\\
&G(m_1,k)=2(2m_1+1)(m_1+m_2+1-k)\geq0,\\
&G(j_k,j_k+1)=2(-(2m_2+1)j_k+m_1+m_2+4m_1m_2)\geq G(m_1,m_1+1)=2m_2(2m_1+1)>0,\\
&G(j_k,j_k+m_2+1)=2(2m_2+1)(m_1-j_k)\geq0.
\end{split}
\end{equation*}
Thus the function $G(j_k,k)$ achieves its minimum $0$ when $j_k=m_1$ and $k=m_1+m_2+1$. And since $\alpha_{m_2,m_1}=2a_{m_1}b_{m_2}>0$, we have $\Upsilon=\left\{V_{1,m_1+m_2+1}\right\}=\left\{\left(2\left(2m_1-m_2\right)+3,2m_1+1\right)\right\}$ and the vector
field associated to $V_{1,m_1+m_2+1}$ is
$$\left(\left(2 m_1+1\right) a_{m_1} b_{m_2} u^{2\left(2m_1-m_2\right)+3} v^{2 m_1},\left(2\left(2m_1+m_2\right)+3\right) a_{m_1} b_{m_2} u^{2\left(2m_1-m_2+1\right)} v^{2 m_1+1}\right).$$	
Hence, the lowest-degree quasi-homogeneous terms of type $\left(2m_1+1, 2\left(2m_1-m_2\right)+1\right)$ for the vector field
\eqref{eq27} are
\begin{scriptsize}
\begin{align*}
	&\left(
	\begin{array}{l}
		\left(2m_1+1\right)a_{m_1}u^2v^{2m_1}\left(a_{m_1} v^{2m_1+1}+b_{m_2}u^{2\left(2m_1- m_2\right)+1}\right)\vspace{1ex}\\
		\left(4m_1\!+\!2m_2\!+\!3\right)a_{m_1} b_{m_2} u^{2\left(2m_1\!-\!m_2\!+\!1\right)}v^{2m_1+1}\!+\!2\left(2 m_1\!+\!1\right)a_{m_1}^2uv^{2\left(2m_1+1\right)}\!+\!\left(2m_2+1\right) \left(b_{m_2}^2\!+\!d_{m_2}^2\right)u^{4\left(2m_1-m_2\right)+3}
	\end{array}
	\right)^T,
\end{align*}
\end{scriptsize}with $a_{m_1}>0, b_{m_2}>0$ and $d_{m_2}>0$. Moreover, the Hamiltonian associated to the  edge of type  $\left(2m_1+1, 2\left(2m_1-m_2\right)+1\right)$ is
\begin{small}
\begin{equation*}
	h_{4\left(2m_1+1\right) \left(2m_1-m_2+1\right)}^{\left(2m_1+1, 2\left(2m_1-m_2\right)+1\right)}\!=\!\frac{2m_2+1}{4\left(2m_1-m_2+1\right)}\left(\left(a_{m_1}v^{2m_1+1}
\!+\!b_{m_2}u^{2\left(2m_1-m_2\right)+1}\right)^2\!+\!\left(d_{m_2}u^{2\left(2m_1-m_2\right)+1}\right)^2\right)u^2
\end{equation*}
\end{small}with $a_{m_1}>0,\ b_{m_2}>0$ and $d_{m_2}>0$.

It is not difficult to check that the Hamiltonians $h_{4\left(m_1+1\right)}^{\left(1,1\right)}$ and  $h_{4\left(2m_1+1\right) \left(2m_1-m_2+1\right)}^{\left(2m_1+1, 2\left(2m_1-m_2\right)+1\right)}$ have no factor of the form $v^{t_1}-\lambda u^{t_2}$ with $\lambda\in\mathbb{R}\setminus\{0\}$. By Theorem \ref{th-11}, the polynomial map $F$ is injective.

\end{proof}

\section*{Acknowledgments}

%The first author would like to express his heartful gratitude to Professor Yulin Zhao from Sun Yat-sen University, who led him into this field.

This research is supported by the National Natural Science Foundation of China (No.11801582, No. 11790273 and No. 11971495)  and Guangdong Basic and Applied Basic Research Foundation (No. 2019A1515011239).

%\bibliographystyle{unsrt}
%\bibliographystyle{elsarticle-num}
%\bibliographystyle{elsarticle-num-names}
%\bibliographystyle{elsarticle-harv}
%\biboptions{square,numbers,sort&compress}
%\bibliographystyle{plain}
%\bibliographystyle{siam}
%\bibliography{ref}

\begin{thebibliography}{100}
	\bibitem{algaba2009integrability}
	{\sc A.~Algaba, E.~Gamero and C.~Garc\'{\i}a}, {\em The integrability problem
		for a class of planar systems}, Nonlinearity, \textbf{22} (2009), ~395--420.
	
	\bibitem{MR2819283}
	{\sc A.~Algaba, C.~Garc\'{\i}a and M.~Reyes}, {\em Characterization of a
		monodromic singular point of a planar vector field}, Nonlinear Anal., \textbf{74}
	(2011), ~5402--5414.
	
	\bibitem{MR0350126}
	{\sc A.~A. Andronov, E.~A. Leontovich, I.~I. Gordon and A.~G. Ma\u{\i}er},
	{\em Qualitative theory of second-order dynamic systems}, Halsted Press , 1973.

	
	\bibitem{bass1982jacobian}
	{\sc H.~Bass, E.~H. Connell and D.~Wright}, {\em The {J}acobian conjecture:
		reduction of degree and formal expansion of the inverse}, Bull. Amer. Math.
	Soc. (N.S.), \textbf{7} (1982), ~287--330.
	
	\bibitem{berezovskaya1993complicated}
	{\sc F.~S. Berezovskaya, N.~B. Medvedeva}, {\em A complicated singular point
		of ``center-focus'' type and the {N}ewton diagram}, Selecta Math., \textbf{13} (1994),~1--15.
	
	\bibitem{MR2552779}
	{\sc F.~Braun, J.~R. dos Santos~Filho}, {\em The real {J}acobian conjecture
		on {$\mathbb R^2$} is true when one of the components has degree 3}, Discrete
	Contin. Dyn. Syst., \textbf{26} (2010), ~75--87.
	
	\bibitem{MR3448779}
	{\sc F.~Braun, J.~Gin\'{e} and J.~Llibre}, {\em A sufficient condition in
		order that the real {J}acobian conjecture in {$\mathbb{R}^2$} holds}, J.
	Differential Equations, \textbf{260} (2016), ~5250--5258.
	
	\bibitem{MR3443401}
	{\sc F.~Braun, J.~Llibre}, {\em A new qualitative proof of a result on the
		real jacobian conjecture}, An. Acad. Brasil. Ci\^{e}nc., \textbf{87} (2015),
	~1519--1524.
	
	\bibitem{Braun2017On}
{\sc F.~Braun, J.~Llibre}, {\em On the connection
		between global centers and global injectivity in the plane},
	\url{arXiv:1706.02643} [math.DS],  (2017).
	
	\bibitem{MR3514314}
	{\sc F.~Braun, B.~Or\'{e}fice-Okamoto}, {\em On polynomial submersions of
		degree 4 and the real {J}acobian conjecture in {$\mathbb{R}^2$}}, J. Math. Anal.
	Appl., \textbf{443} (2016), ~688--706.
	
	\bibitem{vall2021}
	{\sc F.~Braun, C.~Valls}, {\em A weight homogeneous condition to the real
		{J}acobian conjecture in {$\mathbb{R}^2$}}, Proc. Edinb. Math.
	Soc.,  \url{doi:10.1017/S0013091521000766}, (2021).
	
	\bibitem{brunella1990topological}
	{\sc M.~Brunella, M.~Miari}, {\em Topological equivalence of a plane vector
		field with its principal part defined through {N}ewton polyhedra}, J.
	Differential Equations, \textbf{85} (1990), ~338--366.
	
	\bibitem{cima1995global}
	{\sc A.~Cima, A.~Gasull, J.~Llibre and F.~Ma{\~n}osas}, {\em Global
		injectivity of polynomial maps via vector fields}, in Automorphisms of Affine
	Spaces, Springer, 1995, ~105--123.
	
	\bibitem{MR1362759}
	{\sc A.~Cima, A.~Gasull and F.~Ma\~{n}osas}, {\em Injectivity of polynomial
		local homeomorphisms of {${\bf R}^n$}}, Nonlinear Anal., \textbf{26} (1996),~877--885.
	
	\bibitem{cobo2002injectivity}
	{\sc M.~Cobo, C.~Gutierrez and J.~Llibre}, {\em On the injectivity of {$C^1$}
		maps of the real plane}, Canad. J. Math., \textbf{54} (2002), ~1187--1201.
	
	\bibitem{MR714105}
	{\sc L.~M. Dru\.{z}kowski}, {\em An effective approach to {K}eller's {J}acobian
		conjecture}, Math. Ann., \textbf{264} (1983), ~303--313.
	
	\bibitem{MR3866897}
	{\sc A.~Dubouloz, K.~Palka}, {\em The {J}acobian conjecture fails for
		pseudo-planes}, Adv. Math., \textbf{339} (2018), ~248--284.
	
	\bibitem{Dumortier2006qualitative}
	{\sc F.~Dumortier, J.~Llibre and J.~Art{\'e}s}, {\em Qualitative theory of
		planar differential systems}, Springer, 2006.
	
	\bibitem{fernandes2004global}
	{\sc A.~Fernandes, C.~Gutierrez and R.~Rabanal}, {\em Global asymptotic
		stability for differentiable vector fields of {$\mathbb R^2$}}, J. Differential
	Equations, \textbf{206} (2004), ~470--482.
	
	\bibitem{MR4213674}
	{\sc J.~Gin\'{e}, J.~Llibre}, {\em A new sufficient condition in order that
		the real {J}acobian conjecture in {$\mathbb R^2$} holds}, J. Differential
	Equations, \textbf{281} (2021), ~333--340.
	
	\bibitem{MR1839866}
	{\sc J.~Gwo\'{z}dziewicz}, {\em The real {J}acobian conjecture for polynomials
		of degree 3}, vol.~\textbf{76}, 2001, ~121--125.
	
	\bibitem{Jackson}
	{\sc J.~Itikawa, J.~Llibre}, {\em New classes of polynomial maps satisfying
		the real jacobian conjecture in $\mathbb{R}^2$}, An. Acad. Brasil.
	Ci\^{e}nc., \textbf{91} (2019), ~e20170627.
	
	\bibitem{MR4219338}
	{\sc J.~Llibre, C.~Valls}, {\em A sufficient condition for the real
		{J}acobian conjecture in {$\mathbb{R}^2$}}, Nonlinear Anal. Real World Appl., \textbf{60}
	(2021), No. 103298, 10.
	
	\bibitem{medvedeva2006analytic}
	{\sc N.~B. Medvedeva}, {\em On the analytic solvability of the problem of
		distinguishing between a center and a focus}, Tr. Mat. Inst. Steklova, \textbf{254}
	(2006), ~11--100.
	
	\bibitem{MR1292168}
	{\sc S.~Pinchuk}, {\em A counterexample to the strong real {J}acobian
		conjecture}, Math. Z., \textbf{217} (1994), ~1--4.
	
	\bibitem{MR713265}
	{\sc J.~D. Randall}, {\em The real {J}acobian problem}, in Singularities,
	{P}art 2 ({A}rcata, {C}alif., 1981), vol.~\textbf{40}
	Amer. Math. Soc., 1983, ~411--414.
	
	\bibitem{MR714106}
	{\sc K.~Rusek}, {\em A geometric approach to {K}eller's {J}acobian conjecture},
	Math. Ann., \textbf{264 } (1983), ~315--320.
	
	\bibitem{MR1636592}
	{\sc M.~Sabatini}, {\em A connection between isochronous {H}amiltonian centres
		and the {J}acobian conjecture}, Nonlinear Anal., \textbf{34} (1998), ~829--838.
	
	\bibitem{MR1487631}
	{\sc V.~Shpilrain, J.-T. Yu}, {\em Polynomial retracts and the {J}acobian
		conjecture}, Trans. Amer. Math. Soc., \textbf{352} (2000), ~477--484.
	
	\bibitem{MR1631413}
	{\sc S.~Smale}, {\em Mathematical problems for the next century}, Math.
	Intelligencer, \textbf{20} (1998), ~7--15.
	
	\bibitem{tian2021necessary}
	{\sc Y.~Tian, Y.~Zhao}, {\em The necessary and sufficient conditions for the
		real jacobian conjecture}, \url{arXiv:2011.11843v2} [math.DS],  (2021).
	
	\bibitem{van2012polynomial}
	{\sc A.~van~den Essen}, {\em Polynomial automorphisms and the {J}acobian
		conjecture}, Birkh\"{a}user Verlag, Basel, 2000.
	
	\bibitem{MR4242820}
	{\sc A.~van~den Essen, S.~Kuroda and A.~J. Crachiola}, {\em Polynomial
		automorphisms and the {J}acobian conjecture---new results from the beginning
		of the 21st century}, Springer, 2021.
\end{thebibliography}

\end{document}